\documentclass[12pt,reqno]{amsart}

\usepackage[utf8]{inputenc}
\usepackage[T1]{fontenc}
\usepackage[expansion=false]{microtype}

\usepackage{mathtools}          
\usepackage{amssymb,amsthm}
\usepackage{bm}                 

\usepackage[a4paper,margin=1in]{geometry}
\usepackage{enumitem}
\setlist[itemize]{leftmargin=*,itemsep=2pt,topsep=2pt}

\usepackage[dvipsnames,svgnames]{xcolor}
\usepackage[colorlinks=true,breaklinks=true,
            linkcolor=Fuchsia,citecolor=ForestGreen,urlcolor=NavyBlue,
            pdfusetitle]{hyperref}
\usepackage[capitalize,nameinlink,noabbrev]{cleveref}

\theoremstyle{plain}
\newtheorem{theorem}{Theorem}[section]
\newtheorem{proposition}[theorem]{Proposition}
\newtheorem{lemma}[theorem]{Lemma}
\newtheorem{corollary}[theorem]{Corollary}
\newtheorem{theorema}{Theorem}

\theoremstyle{remark}
\newtheorem{remark}[theorem]{Remark}
\crefname{theorema}{Theorem}{Theorems}

\numberwithin{equation}{section}
\allowdisplaybreaks

\newcommand{\R}{\mathbb{R}}
\newcommand{\bS}{\mathbb{S}}
\newcommand{\bD}{\bm{D}}
\newcommand{\bU}{\bm{U}}
\newcommand{\bV}{\bm{V}}
\newcommand{\bE}{\bm{E}}
\newcommand{\bW}{\bm{W}}
\newcommand{\bc}{\bm{c}}
\DeclareMathOperator{\Tr}{Tr}
\DeclareMathOperator{\dist}{dist}
\DeclareMathOperator{\dv}{div}
\newcommand{\cM}{\mathcal{M}}
\newcommand{\cH}{\mathcal{H}}
\newcommand{\cF}{\mathcal{F}}
\newcommand{\cD}{\mathcal{D}}
\newcommand{\cB}{\mathcal{B}}
\newcommand{\cC}{\mathcal{C}}
\newcommand{\cT}{\mathcal{T}}
\newcommand{\cP}{\mathcal{P}}
\newcommand{\cE}{\mathcal{E}}
\newcommand{\cG}{\mathcal{G}}
\newcommand{\cA}{\mathcal{A}}

\begin{document}

\title[Optimal stability of the solenoidal HUP and second order CKN inequalities]{Optimal stability hierarchies of the Heisenberg Uncertainty Principle for solenoidal fields and of the second order Caffarelli--Kohn--Nirenberg inequality}

\author[A. X. Do]{Anh Xuan Do}
\address{Anh Xuan Do: Department of Mathematics, University of Connecticut, Storrs, CT
06269, USA}
\email{anh.do@uconn.edu}
\author[A. T. Duong]{Anh Tuan Duong}
\address{Anh Tuan Duong: Faculty of Mathematics and Informatics, Hanoi University of Science and Technology, Hanoi, Vietnam}
\email{tuan.duonganh@hust.edu.vn}
\author[N. Lam]{Nguyen Lam}
\address{Nguyen Lam: School of Science and the Environment, Grenfell Campus, Memorial
University of Newfoundland, Corner Brook, NL A2H5G4, Canada}
\email{nlam@mun.ca}
\author[G. Lu]{Guozhen Lu}
\address{Guozhen Lu: Department of Mathematics, University of Connecticut, Storrs, CT
06269, USA}
\email{guozhen.lu@uconn.edu}
\author[V. H. Nguyen]{Van Hoang Nguyen}
\address{Van Hoang Nguyen: Department of Mathematics, FPT University, Ha Noi, Vietnam}
\email{vanhoang0610@yahoo.com,~hoangnv47@fe.edu.vn}
\date{\today}

\begin{abstract}In 2018, V. Maz'ya \cite{Maz18} proposed 75 open problems in analysis and PDEs. One of them is about the best constant for the Heisenberg Uncertainty Principle for solenoidal vector fields motivated
by questions in hydrodynamics. This was answered by   Cazacu, Flynn and Lam in dimension two and subsequently  by Hamamoto in all dimensions $N\ge3$ by establishing sharp inequalities with explicit constants. Nevertheless, the much harder stability problem remains open in all dimensions. 
The first main result in this paper is to establish a sharp stability hierarchy by proving that its deficit controls the distance to the set of extremals, with the sharp stability constant $\frac12(N-\sqrt{N^2-4N+12})$ for $N\ge4$ and $1$ for $N=3$, together with chains of remainder terms measuring the distance to explicit larger families of poloidal and toroidal fields. The second main result is the second order $L^2$-Caffarelli--Kohn--Nirenberg inequality with the weights $|x|^{-2a}$ and $|x|^{-2b}$ on the line $1+a+b=0$, for which we obtain the stability constant $\frac12\min\{4(1+a),\sqrt{(N+2a)^2+4N-4}-(N+2a)\}$. This CKN inequality is sharp whenever the second number is the smaller one. We develop a novel method: the fourth order one-dimensional problems attached to the spherical modes are transformed, by a Fourier--Hankel transform of real order, into first order inequalities whose sharp stability is that of weighted Gaussian Poincar\'e inequalities. The inverse transform produces the extremals in terms of Kummer's function. As a byproduct of our new approach, we also obtain substantially simpler proofs, with sharp remainder terms and equality cases, of Hamamoto's one-dimensional inequality and of his sharp uncertainty principle for solenoidal fields. 
\end{abstract}

\maketitle

\section{Introduction and statement of the main results}

\subsection{The Heisenberg Uncertainty Principle and its solenoidal improvement}
Heisenberg's Uncertainty Principle \cite{Hei27} states that the position and the momentum of a particle cannot be measured simultaneously with arbitrary precision. Its mathematical formulation, the Heisenberg--Pauli--Weyl inequality \cite{Wey31, FS97}, asserts that a function and its Fourier transform cannot both be sharply localized. In terms of vector fields on $\R^N$ it reads
\begin{equation}\label{HUP}
\int_{\R^N}|\nabla\bU|^2dx\int_{\R^N}|\bU|^2|x|^2dx\ \ge\ \frac{N^2}{4}\Big(\int_{\R^N}|\bU|^2dx\Big)^2 ,
\end{equation}
where the constant $\frac{N^2}{4}$ is sharp and attained exactly when the components of $\bU$ are proportional to a common Gaussian $e^{-c|x|^2}$; the vector-valued inequality is equivalent to the scalar one, by the Cauchy--Schwarz inequality applied componentwise (see \cite{CFL22}). We refer to \cite{FS97} for a survey of the uncertainty principle, and to \cite{LLLS25, LLR25} for its versions on half-spaces and orthants and with monomial weights.

A non-trivial problem occurs when the test fields in \eqref{HUP} are subject to a differential constraint, since the best constant may then increase. When $\bU$ is curl-free, i.e.\ $\bU=\nabla\phi$ for a scalar potential $\phi$, Cazacu, Flynn and Lam \cite{CFL22} showed that \eqref{HUP} holds with the sharp constant $\frac{(N+2)^2}{4}$, attained when $\phi$ is a Gaussian; this is at the same time a second order uncertainty principle for $\phi$. For solenoidal (namely divergence-free) fields the question is a special case of a problem raised by Maz'ya \cite[Section 3.9]{Maz18}: \emph{find the best value of the constant in \eqref{HUP} when $\bU$ is assumed to be solenoidal.} In dimension two, curl-free and solenoidal fields are isometrically isomorphic (rotate the field by a right angle), so that the result of \cite{CFL22} solves this problem for $N=2$. Solenoidal improvements of functional inequalities go back to Costin and Maz'ya \cite{CM08}, who established the sharp Hardy--Leray inequality
\[
\int_{\R^N}|\nabla\bU|^2dx\ \ge\ \Big(\frac{N-2}{2}\Big)^2\frac{(N+2)^2}{(N+2)^2-8}\int_{\R^N}\frac{|\bU|^2}{|x|^2}dx
\]
for axisymmetric solenoidal fields, improving the constant $(\frac{N-2}{2})^2$ of Leray \cite{Ler33} for unconstrained fields (and of Hardy \cite{Har20} in one dimension); this was refined in a series of papers by Hamamoto and Takahashi \cite{Ham19, Ham20, Ham21, HT20, HT21}, where the axisymmetry was removed, the Rellich--Leray inequality was treated, and the curl-free counterparts were obtained. Since Hardy's inequality can itself be regarded as an uncertainty principle \cite{BV97}, one may expect the same computability for the solenoidal HUP, and this is what Hamamoto proved in \cite{Ham23}. Let us recall the two tools which make the solenoidal problem tractable. The first one is the poloidal--toroidal decomposition: every solenoidal field decaying at infinity splits uniquely as $\bU=\bU_P+\bU_T$, where the toroidal part $\bU_T$ is tangent to the spheres $\{|x|=r\}$ and divergence-free, and the poloidal part $\bU_P$ is generated by a scalar potential through a second order differential operator $\bD$; the two parts, as well as their gradients, are orthogonal on every sphere. The second one is the expansion in spherical harmonics, which reduces each part to one-dimensional radial problems. For the toroidal part these are scalar problems of the first order CKN type, but for the poloidal part they are genuinely of fourth order, and their solution constitutes the core of \cite{Ham23}.

\begin{theorema}[{\cite{Ham23}}]\label{THam}
Let $N\ge3$. For all solenoidal fields $\bU$ with $\int|\nabla\bU|^2+\int|\bU|^2|x|^2+\int|\bU|^2<\infty$,
\begin{equation}\label{HUPsol}
\int_{\R^N}|\nabla\bU|^2dx\int_{\R^N}|\bU|^2|x|^2dx\ \ge\ C_N^2\Big(\int_{\R^N}|\bU|^2dx\Big)^2,\qquad
C_N=\frac{\sqrt{N^2-4N+12}+2}{2},
\end{equation}
and the constant $C_N$ is sharp and attained. More precisely, decomposing $\bU=\bU_P+\bU_T$ into its poloidal and toroidal parts,
\begin{equation}\label{CPCT}
C_N=\min\{C_{P,N},\,C_{T,N}\},\qquad C_{P,N}=\frac{\sqrt{N^2-4N+12}+2}{2},\qquad C_{T,N}=\frac{N+2}{2},
\end{equation}
where $C_{P,N}$ and $C_{T,N}$ are the sharp constants of \eqref{HUPsol} for poloidal and for toroidal fields, respectively. The extremals are poloidal fields of spherical degree one for $N\ge4$, and at $N=3$, where $C_{P,3}=C_{T,3}=\frac52$, additionally the toroidal fields $Axe^{-\lambda|x|^2}$ with $A$ antisymmetric and their sums with poloidal extremals at the same scale.
\end{theorema}

The constant $C_N$ is strictly larger than the unconstrained value $\frac N2$, and strictly smaller than the curl-free value $\frac{N+2}{2}$ for $N\ge4$ (while $C_3=\frac52$), in accordance with the fact that the solenoidal condition is one equation while the curl-free condition consists of $\frac{N(N-1)}{2}$ equations. The proof in \cite{Ham23} rests on the poloidal--toroidal decomposition of Backus \cite{Bac58}, generalized by Weck \cite{Wec97} to differential forms on $\R^N$, on the transformation $\phi\mapsto|x|^\lambda\phi$ of Tertikas and Zographopoulos \cite{TZ07} used in \cite{CFL22}, and, as its main difficulty, on the explicit solution of a one-dimensional minimization problem for a fourth order functional (see \eqref{Hmin} below), for which a shorter proof was later given in \cite{Ham24} and by Huang, Ozawa and Tong \cite{HOT}. To the best of our knowledge, the stability of \eqref{HUPsol} has not been studied so far; one of the reasons is that the extremals are not Gaussians but Kummer functions, so that the Gaussian Poincar\'e inequality, which is the engine of the stability results for \eqref{HUP}, is not directly available for the poloidal part.

\subsection{The second order Caffarelli--Kohn--Nirenberg inequality}
The second problem we consider belongs to the family of the $L^2$-Caffarelli--Kohn--Nirenberg (CKN) inequalities \cite{CKN},
\begin{equation}\label{CKNfirst}
\Big(\int_{\R^N}\frac{|\nabla u|^2}{|x|^{2b}}dx\Big)^{\frac12}\Big(\int_{\R^N}\frac{|u|^2}{|x|^{2a}}dx\Big)^{\frac12}\ \ge\ C(N,a,b)\int_{\R^N}\frac{|u|^2}{|x|^{a+b+1}}dx ,
\end{equation}
which contains the HUP ($a=-1$, $b=0$), the hydrogen uncertainty principle ($a=b=0$) and Hardy's inequality ($a=1$, $b=0$). The sharp constant $C(N,a,b)$ was computed by Costa \cite{Cos08} in a restricted range by the expanding-the-square method, and by Catrina and Costa \cite{CC09} in the full range; alternative direct proofs, together with the classification of all optimizers, were given in \cite{CFL21}. Second order versions, namely CKN inequalities for $\Delta u$ and for curl-free fields, were established by Cazacu, Flynn and Lam \cite{CFL22, CFL23}, and completed by Duong and Nguyen \cite{DN23, DN26}, who computed the sharp constants of the weighted $L^2$-CKN inequalities for curl-free fields and second order derivatives; the sharp stability of the weighted $L^2$-CKN inequalities for curl-free fields is studied in \cite{DDLLNcurl}.

In this paper we consider the second order $L^2$-CKN inequality on the scaling line $1+a+b=0$. Let $N\ge2$, $a>-1$ and $b=-1-a$. Then (see \cite{CFL23, DN26})
\begin{equation}\label{2ndCKN}
\int_{\R^N}\frac{|\Delta f|^2}{|x|^{2a}}dx\int_{\R^N}\frac{|\nabla f|^2}{|x|^{2b}}dx\ \ge\ \Big(\frac{N+4a+2}{2}\Big)^2\Big(\int_{\R^N}|\nabla f|^2dx\Big)^2
\end{equation}
for all $f\in C_0^\infty(\R^N)$, the constant is sharp, and equality holds exactly for $f=c\,\varphi_0(\lambda\,\cdot)$, where
\begin{equation}\label{phi0}
\varphi_0(x)=e^{-\frac{|x|^{2+2a}}{2+2a}},\qquad x\in\R^N .
\end{equation}
The case $a=0$ is the second order HUP $\int|\Delta f|^2\int|x|^2|\nabla f|^2\ge\frac{(N+2)^2}{4}(\int|\nabla f|^2)^2$ of \cite{CFL22}, which, for $\bU=\nabla f$, is the curl-free HUP mentioned above. It is worth noting that for $a\ne0$ the inequality \eqref{2ndCKN}, which involves $\Delta f$, is different from the CKN inequality for the curl-free field $\nabla f$, which involves the full Hessian and whose stability is studied in \cite{DDLLNcurl}; the two coincide exactly at $a=0$, and for radial $f$ one computes that $\int|\nabla^2f|^2|x|^{-2a}-\int|\Delta f|^2|x|^{-2a}=-2a(N-1)|\bS^{N-1}|\int_0^\infty|f'|^2r^{N-3-2a}dr$, so that, on radial functions, the Hessian integral is the smaller one for $a>0$ and the larger one for $a<0$. The reason why we single out the line $1+a+b=0$ is structural: on this line the spherical modes of \eqref{2ndCKN} reduce exactly to the one-dimensional problem of Hamamoto, whereas off the line (and for the Hessian version) an additional lower order term appears; see the identity \eqref{keyiden} below, where this term carries the factor $1+a+b$.

\subsection{Sharp stability}
In \cite{BL85}, Brezis and Lieb raised the question whether the deficit of the sharp Sobolev inequality of the first order controls the distance to the manifold of optimizers. This question, answered affirmatively by Bianchi and Egnell \cite{BE91}, gave rise to a broad research program on the stability of functional and geometric inequalities; see for instance \cite{BDNN20, BGKM25, Carlen, CF13, CFMP09, CLT-JFA, CLT24, CLT242, CLT243, CLTW, DEFFL, FN19, FZ22, LLR26, VHN16, VHN19}, a list which is far from exhaustive. Since the proof in \cite{BE91} is by contradiction, it gives no information on the stability constant, and the determination of sharp or explicit stability constants has become a central theme: for the Sobolev inequality the local constant is the spectral gap $\frac{4}{N+4}$ \cite{CFW13}, the global constant is strictly smaller and attained \cite{Kon23}, and explicit lower bounds with the optimal dimensional dependence were obtained in a pioneering work by Dolbeault, Esteban, Figalli, Frank and Loss  in \cite{DEFFL}, subsequently by Chen, Lu and Tang in \cite{CLT24, CLT242, CLT243, CLTW} for the Hardy--Littlewood--Sobolev, higher order and fractional Sobolev inequalities and on the Heisenberg group.

For the HUP, stability was first obtained by McCurdy and Venkatraman \cite{MV21} by concentration-compactness, without constants, then with explicit but non-optimal constants by Fathi \cite{Fat21}, and finally in sharp form in \cite{CFLL24}: the product deficit of \eqref{HUP} (for scalar $u$) is bounded below by the $L^2$-distance to the Gaussians with the sharp constant $1$. The approach of \cite{CFLL24}, based on an exact remainder identity and on the sharp Gaussian Poincar\'e inequality, was extended to the whole family \eqref{CKNfirst} in \cite{DLLN26}, where four of the authors computed the sharp stability constants for all admissible parameters through a family of weighted Poincar\'e inequalities with Gaussian-type measures (see also \cite{DFLL23} for the $L^p$ theory, \cite{LLLS26} for half-spaces and orthants, and \cite{LLR25, LLR26} for related developments). The stability of the second order HUP, i.e.\ of \eqref{2ndCKN} at $a=0$, was first obtained by Duong and Nguyen \cite{DN25}, and the sharp stability was established by Do, Lam and Lu \cite{DoLL26}, who treated both the second order and the curl-free field cases, see also \cite{HY25}; in \cite{DLLNfrac}, four of the authors gave a short proof of these results, and treated the fractional HUPs, through a Fourier dictionary between the second order and fractional inequalities and \eqref{CKNfirst}.

Let us make precise what is meant by sharp stability. For an inequality $\mathcal L(u)\ge C_*\,\mathcal R(u)$ with a set $\mathcal M$ of extremals, a stability estimate is a lower bound of the deficit $\mathcal L(u)-C_*\mathcal R(u)$ by $\kappa\,\dist(u,\mathcal M)^2$ for a suitable norm; the estimate is sharp if $\kappa$ is the largest possible constant, and the best one can hope for is $\kappa$ explicit and attained, or approached, by explicit test functions. Two methods have been used to obtain sharp constants for inequalities of CKN type. The first one, in \cite{CFLL24, DLLN26, DoLL26}, writes the deficit as an exact remainder identity (an integral of a square), then substitutes $u=\varphi_0v$ with $\varphi_0$ the extremal, and applies a Poincar\'e inequality for a Gaussian-type measure; it is also the method by which the sharp constant of \eqref{HUPsol} was reproved in \cite{Ham24, HOT}, the square being the one in \eqref{RemainderIdt} below (the original proof in \cite{Ham23} is by an expansion in Laguerre polynomials). The second one, in \cite{DLLNfrac}, is a Fourier dictionary translating the second order and fractional HUPs into first order CKN inequalities. The present paper combines the two: the exact square of \eqref{RemainderIdt} is transformed by a Fourier--Hankel transform into a first order square, and only then is the Gaussian Poincar\'e inequality applied.

Let us also mention that, for the inequalities with Gaussian extremals, the stability question is closely related to the spectral theory of the Ornstein--Uhlenbeck and Laguerre operators, whose eigenvalues are equally spaced; this is what makes it possible to obtain not only a first stability estimate but whole chains of remainder terms, each level of the expansion contributing one term. We refer to \cite{CFLL24, DoLL26} for stability estimates of the HUP and of its second order version obtained in this way, and to \cite{LL26} for complete hierarchies of stability estimates for the $L^2$-Poincar\'e inequalities on Euclidean balls and for the Gaussian Poincar\'e inequality.

The purpose of this paper is to establish the sharp stability, with explicit constants, of the solenoidal HUP \eqref{HUPsol} and of the second order CKN inequality \eqref{2ndCKN} on the whole line $1+a+b=0$, in the dilation invariant product form of Bianchi--Egnell (for \eqref{2ndCKN}, the sharpness of the product-form constant is established for $a\ge a_*(N)$, see Theorem \ref{thm:prod2nd}), together with chains of remainder terms of a new type. Along the way we obtain a new proof of Theorem \ref{THam}, i.e.\ of the answer to Maz'ya's problem, in which the poloidal blocks are treated by a Fourier--Hankel transform instead of the Laguerre expansions of \cite{Ham23}, and which produces the inequality directly with its sharp remainder terms. As in \cite{CFLL24, DLLN26}, the product form is reached in two steps: we first fix the scale and prove stability for the deficit in sum form, which is technically simpler and in fact yields more (sharp chains of remainder terms), and we then transfer the estimates to the product form by a dilation. Our method is different from the ones in the works quoted above, and we believe that it is of independent interest: the one-dimensional problems attached to the spherical modes are transformed by a Fourier--Hankel transform into first order CKN problems, for which the sharp stability of \cite{DLLN26} is available; as a byproduct we obtain a new proof, with a sharp remainder term, of the one-dimensional inequality of \cite{Ham23}.

\subsection{Main results}\label{sec:main}
Our first group of results concerns the solenoidal HUP \eqref{HUPsol}. The natural deficit of \eqref{HUPsol}, in the sense of Bianchi--Egnell, is the \emph{product deficit} (the index $1$ refers to the product form, the index $2$ below to the sum form)
\begin{equation}\label{proddef}
\delta_{1,N}(\bU):=\Big(\int_{\R^N}|\nabla\bU|^2dx\Big)^{\frac12}\Big(\int_{\R^N}|\bU|^2|x|^2dx\Big)^{\frac12}-C_N\int_{\R^N}|\bU|^2dx\ \ge0 ,
\end{equation}
which is invariant under the dilations $\bU\mapsto\lambda^{N/2}\bU(\lambda\,\cdot)$, $\lambda>0$, and vanishes exactly on the extremals of \eqref{HUPsol} at every scale; the stability question is whether $\delta_{1,N}(\bU)$ controls the $L^2$-distance from $\bU$ to the set of extremals. For the statements we need Hamamoto's extremal profiles and two finite dimensional families of toroidal fields. For $k\ge1$ let $\cH_k$ denote the space of spherical harmonics of degree $k$ on $\bS^{N-1}$, and put
\begin{equation}\label{fk}
f_k(r):=r^{k-1}e^{-\frac{r^2}{2}}\,{}_1F_1\Big(\frac{N+2k-\sqrt{(N+2k)^2-8(N+k-2)}}{4},\ \frac{N+2k}{2},\ \frac{r^2}{2}\Big),
\end{equation}
\begin{equation}\label{CNk}
C(N,k):=1+\frac{\sqrt{(N+2k)^2-8(N+k-2)}}{2},
\end{equation}
so that $C(N,1)=C_{P,N}$, and denote by $\bD$ the poloidal generator of \cite{Ham23}, recalled in Section \ref{sec:sol}, which maps a scalar function $f$ with vanishing spherical means to a poloidal field $\bD f$. The fields $\bD(f_k\phi)$, $\phi\in\cH_k$, are exactly the extremals of \eqref{HUPsol} restricted to poloidal fields of degree $k$ (with the constant $C(N,k)$ in place of $C_N$) at the scale carrying the Gaussian factor $e^{-|x|^2/2}$. Let $\mathcal A_N$ be the set of antisymmetric $N\times N$ matrices and let $\cM_N$ be the set of all $N$-tuples $\cB=(B^1,\dots,B^N)$ of symmetric $N\times N$ matrices such that
\begin{equation}\label{MN}
\Tr B^k=0,\qquad \sum_{i=1}^NB^i_{ik}=0,\qquad B^i_{jl}+B^j_{il}+B^l_{ij}=0\qquad\text{for all }i,j,k,l\in\{1,\dots,N\};
\end{equation}
we write $\cB x\cdot x:=(B^1x\cdot x,\dots,B^Nx\cdot x)$. We show in Lemma \ref{lem:MN} that $(Ax+\cB x\cdot x)e^{-|x|^2/2}$ is a toroidal field if and only if $A\in\mathcal A_N$ and $\cB\in\cM_N$; these are the toroidal fields whose profiles are polynomials of degree at most two. For a finite dimensional subspace $\cE$ of $L^2(\R^N;\R^N)$ we denote by
\[
\dist(\bU,\cE):=\inf_{\bE\in\cE}\Big(\int_{\R^N}|\bU-\bE|^2dx\Big)^{\frac12}=\|\bU-\Pi_{\cE}\bU\|_{L^2(\R^N)}
\]
the $L^2$-distance from $\bU$ to $\cE$, where $\Pi_\cE$ is the orthogonal projection onto $\cE$; if $\cE$ is a union of dilation orbits, as in \eqref{dil} below, the infimum is over all elements of $\cE$. We introduce the families
\begin{align*}
\cE_1&:=\{\bD(f_1\phi_1):\phi_1\in\cH_1\},&
\cE_2&:=\{\bD(f_1\phi_1)+\bD(f_2\phi_2):\phi_1\in\cH_1,\ \phi_2\in\cH_2\},\\
\cT_1&:=\{Axe^{-|x|^2/2}:A\in\mathcal A_N\},&
\cT_2&:=\{(Ax+\cB x\cdot x)e^{-|x|^2/2}:A\in\mathcal A_N,\ \cB\in\cM_N\},
\end{align*}
so that $\cE_1\subset\cE_2$ are poloidal and $\cT_1\subset\cT_2$ are toroidal, together with their dilation orbits
\begin{equation}\label{dil}
\cE^{\rm dil}:=\big\{\lambda^{N/2}\bE(\lambda\,\cdot):\bE\in\cE,\ \lambda>0\big\},
\end{equation}
which consist of the same fields with the Gaussian factor replaced by its dilates, normalized in $L^2$. By Theorem \ref{THam}, $\cE_1^{\rm dil}$ (for $N\ge4$) and $(\cE_1+\cT_1)^{\rm dil}$ (for $N=3$) are exactly the zero sets of $\delta_{1,N}$, i.e.\ the full sets of extremals of \eqref{HUPsol}. Finally we set
\begin{align*}
\Lambda_1&:=N-\sqrt{N^2-4N+12},&\Lambda_2&:=\sqrt{N^2+16}-N,\\
\Lambda_3&:=N+2-\sqrt{N^2+16},&\Lambda_4&:=\sqrt{N^2-4N+12}-N+2 ,
\end{align*}
which are positive and add up to $4$.

\begin{theorem}\label{thm:prodHUP4}
Let $N\ge4$. Then for all solenoidal fields $\bU$,
\begin{align}
\delta_{1,N}(\bU)&\ \ge\ \tfrac12\Big[\Lambda_1\dist(\bU,\cE_1^{\rm dil})^2+\Lambda_2\dist(\bU,(\cE_1+\cT_1)^{\rm dil})^2\notag\\
&\qquad\qquad+\Lambda_3\dist(\bU,(\cE_2+\cT_1)^{\rm dil})^2+\Lambda_4\dist(\bU,(\cE_2+\cT_2)^{\rm dil})^2\Big],\label{prod4}
\end{align}
and the constant $\frac12\Lambda_1=\frac12\big(N-\sqrt{N^2-4N+12}\big)$ is sharp. Moreover, the distances may be taken to the families dilated by the single balancing parameter $\lambda_*=\lambda_*(\bU)$ of Section \ref{sec:prod}, which is a stronger statement.
\end{theorem}

\begin{theorem}\label{thm:prodHUP3}
Let $N=3$. Then for all solenoidal fields $\bU$ on $\R^3$,
\begin{equation}\label{prod3}
\delta_{1,3}(\bU)\ \ge\ \dist(\bU,(\cE_1+\cT_1)^{\rm dil})^2+\dist(\bU,(\cE_2+\cT_2)^{\rm dil})^2 ,
\end{equation}
and the constant $1$ is sharp; the same stronger statement with a single balancing parameter holds.
\end{theorem}

The first terms of \eqref{prod4} and \eqref{prod3} are the sharp stability estimates of \eqref{HUPsol} in the sense of Bianchi--Egnell \cite{BE91}: the product deficit controls, with an explicit sharp constant, the $L^2$-distance to the full set of extremals. The remaining terms are remainder terms of a new type: they show that the deficit still controls the distance to larger families, with explicit constants; these families are built from the next poloidal extremal profile $f_2$ and from the toroidal fields with polynomial profiles of degree at most two, which correspond to the second level of the Hermite expansion. At $N=3$ the family $(\cE_1+\cT_1)^{\rm dil}$ consists of the superpositions of a poloidal and a toroidal extremal \emph{at the same scale}. Superpositions with two different scales are not extremals, and they are controlled by \eqref{prod3} as well, since the right-hand side does not vanish on them.

Theorems \ref{thm:prodHUP4} and \ref{thm:prodHUP3} are deduced, by a dilation, from stability estimates at a fixed scale, which we now state. They concern the deficit of \eqref{HUPsol} in \emph{sum form},
\begin{equation}\label{deficit}
\delta_{2,N}(\bU):=\int_{\R^N}|\nabla \bU|^2dx+\int_{\R^N}|\bU|^2|x|^2dx-2C_N\int_{\R^N}|\bU|^2dx\ \ge0 ,
\end{equation}
which is not dilation invariant and vanishes exactly on the extremals at one scale, the one carrying the Gaussian factor $e^{-|x|^2/2}$; one has $\delta_{2,N}\ge2\delta_{1,N}$, with equality when the two integrals are balanced, and this is how the product form is recovered. The sum form is technically simpler, since the scale is fixed once and for all, and it yields more: the whole chains below are sharp.

\begin{theorem}\label{thm:HUP4}
Let $N\ge4$. Then for all solenoidal fields $\bU$,
\begin{align}\label{stabHUP4}
\delta_{2,N}(\bU)\ \ge\ &\Lambda_1\dist(\bU,\cE_1)^2+\Lambda_2\dist(\bU,\cE_1+\cT_1)^2\notag\\
&+\Lambda_3\dist(\bU,\cE_2+\cT_1)^2+\Lambda_4\dist(\bU,\cE_2+\cT_2)^2 .
\end{align}
All four constants are sharp, in the following sense: for each $j\in\{1,2,3,4\}$ there is a solenoidal field for which equality holds in \eqref{stabHUP4} with the $j$-th term nonzero, so that no constant can be increased while the others are kept fixed.
\end{theorem}

\begin{theorem}\label{thm:HUP3}
Let $N=3$, so that $2C_3=5$. Then for all solenoidal fields $\bU$ on $\R^3$,
\begin{equation}\label{stabHUP3}
\delta_{2,3}(\bU)\ \ge\ 2\dist(\bU,\cE_1+\cT_1)^2+2\dist(\bU,\cE_2+\cT_2)^2 .
\end{equation}
Both constants are sharp in the same sense as in Theorem \ref{thm:HUP4}.
\end{theorem}

Since the families in \eqref{stabHUP4} and \eqref{stabHUP3} are nested, the distances decrease along the chains, and the sum of the constants is $4$ in both cases. Discarding the smaller distances we obtain the following estimate with a single distance, to the largest family.

\begin{corollary}\label{cor:single}
Let $N\ge3$. For all solenoidal fields $\bU$,
\begin{equation}\label{single}
\delta_{2,N}(\bU)\ \ge\ 4\,\dist(\bU,\cE_2+\cT_2)^2 ,
\end{equation}
and the constant $4$ is sharp. Consequently,
\begin{equation}\label{singleprod}
\delta_{1,N}(\bU)\ \ge\ 2\,\dist\big(\bU,(\cE_2+\cT_2)^{\rm dil}\big)^2 ,
\end{equation}
and more precisely $\delta_{1,N}(\bU)\ge2\,\dist(\bU,(\cE_2+\cT_2)_{1/\lambda_*})^2$ with the balancing parameter $\lambda_*=\lambda_*(\bU)$ of Section \ref{sec:prod}.
\end{corollary}

The estimate \eqref{single} says that the sum-form deficit controls, with the constant $4$, the $L^2$-distance to the explicit finite dimensional family $\cE_2+\cT_2$, of dimension $N+\dim\cH_2+\dim\cA_N+\dim\cM_N$; the constant $4$ is the gap inside the optimal block, and it is attained by the second level of the first poloidal block, which is orthogonal to $\cE_2+\cT_2$ and has deficit exactly $4$ times its squared norm. For $N\ge4$, \eqref{single} follows from \eqref{stabHUP4} since $\dist(\bU,\cF)\ge\dist(\bU,\cE_2+\cT_2)$ for each of the four families $\cF$ and $\Lambda_1+\Lambda_2+\Lambda_3+\Lambda_4=4$; for $N=3$ it follows from \eqref{stabHUP3} in the same way, with $2+2=4$. The product form \eqref{singleprod} follows from \eqref{single} by the dilation argument of Section \ref{sec:prod}, with half the constant; we do not know whether the constant $2$ is sharp there, since the test field realizing the constant $4$ in \eqref{single} is not orthogonal to the dilates of $\cE_2+\cT_2$.

In Theorems \ref{thm:HUP4} and \ref{thm:HUP3} the family $\cE_1$ (for $N\ge4$), respectively $\cE_1+\cT_1$ (for $N=3$), is exactly the set of extremals of \eqref{HUPsol} at the scale fixed by the deficit \eqref{deficit}. It is worth noting that the four constants of \eqref{stabHUP4} add up to $4$, which is the gap inside the optimal block, and that the whole chain is sharp. The toroidal part of Theorems \ref{thm:HUP4} and \ref{thm:HUP3} rests on the following stability estimates, which we state separately since they are of independent interest. Recall that $C_{T,N}=\frac{N+2}{2}$.

\begin{theorem}\label{thm:tor}
Let $N\ge3$. Then for all toroidal fields $\bU$,
\begin{align}\label{stabtor}
\int_{\R^N}|\nabla \bU|^2dx+\int_{\R^N}|\bU|^2|x|^2dx&-(N+2)\int_{\R^N}|\bU|^2dx\notag\\
\ \ge\ &2\inf_{A\in\mathcal A_N}\int_{\R^N}\big|\bU-Axe^{-\frac{|x|^2}{2}}\big|^2dx\notag\\
&+2\inf_{\substack{A\in\mathcal A_N\\ \cB\in\cM_N}}\int_{\R^N}\big|\bU-(Ax+\cB x\cdot x)e^{-\frac{|x|^2}{2}}\big|^2dx .
\end{align}
The constant $2$ in the first term is sharp.
\end{theorem}

Our second group of results concerns \eqref{2ndCKN}. Its product deficit is
\begin{equation}\label{proddef2}
\delta_{1,a}(f):=\Big(\int_{\R^N}\frac{|\Delta f|^2}{|x|^{2a}}dx\Big)^{\frac12}\Big(\int_{\R^N}\frac{|\nabla f|^2}{|x|^{2b}}dx\Big)^{\frac12}-\frac{N+4a+2}{2}\int_{\R^N}|\nabla f|^2dx\ \ge0 ,
\end{equation}
which is invariant under $f\mapsto\lambda^{N/2-1}f(\lambda\,\cdot)$ and vanishes exactly on the extremals $c\,\varphi_0(\lambda\,\cdot)$ of \eqref{2ndCKN}. For $k\ge1$ we write
\begin{equation}\label{C1}
C_k(N,a):=\sqrt{(N+2a)^2+4k^2+4k(N-2)}-(N+2a),
\end{equation}
so that $C_1(N,a)=\sqrt{(N+2a)^2+4N-4}-(N+2a)$. Let $s(x):=\frac{|x|^{2+2a}}{2+2a}$ and let $h_1$ be the Kummer function
\begin{equation}\label{h1}
h_1(s):=e^{-s}\,{}_1F_1(b_1,\mu_1,s),\quad
\mu_1:=\frac{N+2+2a}{2+2a},\quad\epsilon_1:=\frac{1}{2+2a},\quad b_1:=\frac{\mu_1-\sqrt{\mu_1^2-4\epsilon_1}}{2},
\end{equation}
which is the profile of the first spherical mode at the scale of \eqref{phi0} (see Section \ref{sec:2nd}). We introduce the two families of gradient fields
\begin{equation}\label{Efam}
\begin{aligned}
\mathcal G_1&:=\{c\,\nabla\varphi_0:c\in\R\},\\
\mathcal G_2&:=\Big\{\nabla\Big[\big(c_0+c_1|x|^{2+2a}\big)\varphi_0+|x|\,h_1(s(x))\,\phi_1\Big(\tfrac{x}{|x|}\Big)\Big]:c_0,c_1\in\R,\ \phi_1\in\cH_1\Big\},
\end{aligned}
\end{equation}
so that $\mathcal G_1\subset\mathcal G_2$, and, with the notation \eqref{dil}, their dilation orbits $\cG_1^{\rm dil}\subset\cG_2^{\rm dil}$; in particular $\cG_1^{\rm dil}$ is exactly the set of gradients of the extremals $c\,\varphi_0(\lambda\,\cdot)$ of \eqref{2ndCKN}. We set
\[
\Theta_1:=\min\big\{4(1+a),\,C_1(N,a)\big\},\quad
\Theta_2:=\min\big\{8(1+a),\ C_1(N,a)+4(1+a),\ C_2(N,a)\big\}-\Theta_1 .
\]

\begin{theorem}\label{thm:prod2nd}
Let $N\ge2$, $a>-1$, $b=-1-a$, and let $f\in C_0^\infty(\R^N)$. Then
\begin{equation}\label{prod2nd}
\delta_{1,a}(f)\ \ge\ \frac{\Theta_1}{2}\dist(\nabla f,\cG_1^{\rm dil})^2+\frac{\Theta_2}{2}\dist(\nabla f,\cG_2^{\rm dil})^2 ,
\end{equation}
and the same stronger statement with a single balancing parameter holds. If $a\ge a_*(N):=\frac{-(N+6)+\sqrt{N^2+4N-4}}{8}$, then $\Theta_1=C_1(N,a)$ and the constant $\frac{\Theta_1}{2}=\frac12C_1(N,a)$ is sharp; in particular the first term of \eqref{prod2nd} is then the sharp stability estimate of \eqref{2ndCKN} in the sense of Bianchi--Egnell.
\end{theorem}

As before, Theorem \ref{thm:prod2nd} is deduced from a stability estimate at a fixed scale, for the deficit in sum form
\begin{equation}\label{deficit2}
\delta_{2,a}(f):=\int_{\R^N}\frac{|\Delta f|^2}{|x|^{2a}}dx+\int_{\R^N}\frac{|\nabla f|^2}{|x|^{2b}}dx-(N+4a+2)\int_{\R^N}|\nabla f|^2dx\ \ge0 ,
\end{equation}
which satisfies $\delta_{2,a}\ge2\delta_{1,a}$ and vanishes exactly on the extremals $c\,\varphi_0$ at the scale of \eqref{phi0}.

\begin{theorem}\label{thm:2nd}
Let $N\ge2$, $a>-1$ and $b=-1-a$. Then for all $f\in C_0^\infty(\R^N)$,
\begin{equation}\label{stab2nd}
\delta_{2,a}(f)\ \ge\ \Theta_1\,\dist(\nabla f,\mathcal G_1)^2+\Theta_2\,\dist(\nabla f,\mathcal G_2)^2 .
\end{equation}
Both constants are sharp: there are functions for which equality holds in \eqref{stab2nd} with only the first term nonzero, and functions for which equality holds with the second term nonzero, so that none of $\Theta_1,\Theta_2$ can be increased while the other is kept fixed. In particular
\begin{equation}\label{stab2nd1}
\delta_{2,a}(f)\ \ge\ \min\big\{4(1+a),\,C_1(N,a)\big\}\inf_{c\in\R}\int_{\R^N}|\nabla f-c\,\nabla\varphi_0|^2dx ,
\end{equation}
which is the sharp stability estimate of \eqref{2ndCKN} at a fixed scale.
\end{theorem}

It is worth noting that the competitors in the constants have different origins: $4(1+a)$ and $8(1+a)$ are the first two levels of the radial problem, i.e.\ of a weighted Gaussian Poincar\'e inequality of \cite{DLLN26}; $C_k(N,a)$ is the gap between the $k$-th spherical mode and the radial one, computed through Hamamoto's one-dimensional inequality (Theorem \ref{thm:1d} below), and is increasing in $k$; and $C_1(N,a)+4(1+a)$ is the second level inside the first spherical mode, which comes from the remainder term in Theorem \ref{thm:1d}. A direct computation shows that $4(1+a)\le C_1(N,a)$ if and only if $a\le a_*(N)\in(-1,-\frac12)$, with $a_*(N)$ as in Theorem \ref{thm:prod2nd}: the radial gap is the smaller one near the Hardy endpoint $a=-1$, and the first spherical mode is the smaller one for $a\ge a_*(N)$. In particular for $a=0$ the sharp constant is $\sqrt{N^2+4N-4}-N$ for all $N\ge2$.

Both groups of results rest on a one-dimensional inequality of Hamamoto \cite[Theorem 5.2]{Ham23} for a fourth order functional. We record in Section \ref{sec:pre} a two-parameter extension of it obtained by expanding a square (Proposition \ref{prop:ext}), and we give a new proof of it, entirely different from his, which provides in addition a sharp remainder term, the equality cases, and the whole sequence of critical values of the associated quotient (Theorem \ref{thm:1d}). We find it remarkable that the Fourier--Hankel transform reduces Hamamoto's fourth order minimization problem to a weighted Gaussian Poincar\'e inequality in a few lines.

\subsection{Notation and conventions}
Vector fields are denoted by bold letters, $\bU,\bV,\bE$, and their components by $U_i,V_i,E_i$; points of $\R^N$ are denoted by $x$, with $r=|x|$ and $\sigma=x/r\in\bS^{N-1}$, and the dual variable of the Fourier transform by $\xi$. For a radial function we use the same letter for the function and its profile. All functions and fields are real-valued; since all the functionals we consider are quadratic with real coefficients, the results extend at once to complex-valued functions by applying them to the real and imaginary parts, and we write $|f|^2$, $|f'|^2$ throughout. The solenoidal fields are always assumed to be smooth and to satisfy $\int|\nabla\bU|^2+\int|\bU|^2|x|^2+\int|\bU|^2<\infty$, or to belong to the completion of such fields for the corresponding norm; this is the space in which Theorem \ref{THam} holds \cite{Ham23}, and all the identities below extend to it by density. Constants such as $C,c$ may change from line to line, while the constants $C_N$, $C(N,k)$, $C_k(N,a)$, $\Lambda_j$, $\Theta_j$ are fixed by the definitions of Section \ref{sec:main}.

\subsection{Strategy}\label{sec:strategy}
We work with the deficits in sum form, i.e.\ at a fixed scale; the product form is recovered by applying the fixed-scale estimates to the dilate of the test field which balances the two integrals (Lemma \ref{lem:transfer}). Both problems reduce, after a decomposition into spherical harmonics (and, for \eqref{HUPsol}, the poloidal--toroidal decomposition), to one-dimensional inequalities of the form \eqref{RemainderIdt}, whose remainder we need to control. The novelty of our approach is the following observation. When $2\mu$ is an integer and $f$ is regarded as a radial function $G$ on $\R^{2\mu}$, the three integrals in \eqref{QP} (with $x=r^2$) are, up to constants, $\int|\Delta G|^2$, $\int|\nabla G|^2|X|^2-4\epsilon\int |G|^2$ and $\int|\nabla G|^2$. By Plancherel's identity these become $\int|\xi|^4|\widehat G|^2$, $\int|\xi|^2|\nabla\widehat G|^2-4\epsilon\int|\widehat G|^2$ and $\int|\xi|^2|\widehat G|^2$, and after extracting the power $|\xi|^\alpha$ with $\alpha(2\mu+\alpha)=-4\epsilon$ the zero order term disappears: Hamamoto's inequality is turned into a \emph{first order} CKN inequality \eqref{CKNfirst} with weights, for radial functions on $\R^{2\mu}$, whose parameters satisfy $1+b-a=2$. For radial functions the latter is equivalent to a weighted Gaussian Poincar\'e inequality, whose sharp stability was established in \cite{DLLN26} with the constant $2(1+b-a)=4$, and the distance in Fourier variables is transformed back into the distance to Hamamoto's extremals by an explicit computation of an inverse Fourier transform in terms of Kummer's function. This gives, for each spherical mode, the block estimate ``deficit $\ge4\,\dist^2$'' with the sharp constant $4$, from which the theorems follow by comparing the gaps between the blocks. Since only radial functions are involved, the dimension $2\mu$ need not be an integer, and the same argument, carried out directly with the Hankel transform of real order in Section \ref{sec:pre}, gives Theorem \ref{thm:1d}.

The paper is organized as follows. In Section \ref{sec:pre} we collect the weighted Gaussian Poincar\'e inequality, we state and prove the two-parameter extension of Hamamoto's inequality and its sharp remainder version (Theorem \ref{thm:1d}), and we collect the change-of-functions identities we need. Section \ref{sec:sol} recalls the poloidal--toroidal decomposition and establishes the Fourier reduction of the poloidal blocks. Section \ref{sec:tor} is devoted to toroidal fields and to the proof of Theorem \ref{thm:tor}. In Section \ref{sec:proofHUP} we combine the blocks and prove the fixed-scale Theorems \ref{thm:HUP4} and \ref{thm:HUP3}, and then transfer them to the product form, proving Theorems \ref{thm:prodHUP4} and \ref{thm:prodHUP3}. Finally, in Section \ref{sec:2nd}, we prove Theorems \ref{thm:2nd} and \ref{thm:prod2nd}.

\section{Preliminaries}\label{sec:pre}

\subsection{The weighted Gaussian Poincar\'e inequality}
We recall the weighted Gaussian Poincar\'e inequality for radial functions of \cite{DLLN26}, in the one-dimensional form we need.

\begin{theorema}[{\cite[Theorems 1.2 and 1.4]{DLLN26}}]\label{TA}
Let $\alpha>0$ and let $v\in W^{1,2}\big((0,\infty),e^{-t^2/2}t^{\alpha-1}dt\big)$. Then one has
\begin{align*}
\int_0^\infty|v'|^2e^{-\frac{t^2}{2}}t^{\alpha-1}dt\ \ge\ &2\inf_{c\in\R}\int_0^\infty|v-c|^2e^{-\frac{t^2}{2}}t^{\alpha-1}dt\\
&+2\inf_{c_0,c_1\in\R}\int_0^\infty|v-c_0-c_1t^2|^2e^{-\frac{t^2}{2}}t^{\alpha-1}dt ,
\end{align*}
and both constants are sharp: the first is attained exactly by $v=a_0+a_1t^2$, the second exactly by the functions $v=a_0+a_1t^2+a_2t^4$ with $a_2\ne0$.
\end{theorema}

The two-term form follows from the proof in \cite[Section 2]{DLLN26}, which is by expansion in Laguerre polynomials; since the same computation is used, in a different guise, in the proof of Theorem \ref{thm:1d} below, let us give it. With $z=t^2/2$ and $u(z)=v(t)$, so that $v'(t)=t\,u'(z)$ and $t^{\alpha+1}dt=(2z)^{\alpha/2}dz$, one has
\begin{align*}
\int_0^\infty|v'|^2e^{-\frac{t^2}{2}}t^{\alpha-1}dt&=2^{\alpha/2}\int_0^\infty|u'|^2z^{\alpha/2}e^{-z}dz,\\
\int_0^\infty|v|^2e^{-\frac{t^2}{2}}t^{\alpha-1}dt&=2^{\alpha/2-1}\int_0^\infty|u|^2z^{\alpha/2-1}e^{-z}dz .
\end{align*}
Expanding $u=\sum_{j\ge0}a_jL_j^{(\alpha/2-1)}(z)$ in the generalized Laguerre polynomials, which are orthogonal in $L^2(z^{\alpha/2-1}e^{-z}dz)$ with $\int_0^\infty|L_j^{(\beta)}|^2z^\beta e^{-z}dz=\frac{\Gamma(j+\beta+1)}{j!}$, and using $\frac{d}{dz}L_j^{(\beta)}=-L_{j-1}^{(\beta+1)}$ \cite[Chapter V]{Szego}, one gets
\[
\int_0^\infty|u'|^2z^{\alpha/2}e^{-z}dz=\sum_{j\ge1}j\,|a_j|^2\frac{\Gamma(j+\frac\alpha2)}{j!},\quad
\int_0^\infty|u|^2z^{\alpha/2-1}e^{-z}dz=\sum_{j\ge0}|a_j|^2\frac{\Gamma(j+\frac\alpha2)}{j!},
\]
so that, with the factor $2^{\alpha/2}/2^{\alpha/2-1}=2$, the level $j$ contributes $2j$ times its squared norm to the left-hand side, while the two infima on the right-hand side are the squared norms of the projections onto the levels $j\ge1$ and $j\ge2$; since $2j\ge2$ for $j\ge1$ and $2j\ge4$ for $j\ge2$, the inequality follows, with equality exactly when $a_j=0$ for $j\ge3$, i.e.\ $v=a_0+a_1t^2+a_2t^4$.

By an elementary change of variable, Theorem \ref{TA} takes the following two forms, which are the ones we shall use.

\begin{corollary}\label{cor:GP}
The following two inequalities hold.
\begin{itemize}
\item[(i)] For $\alpha>0$ and all $v$,
\begin{equation}\label{GP1}
\int_0^\infty|v'|^2r^{\alpha-1}e^{-r^2}dr\ \ge\ 4\inf_{c\in\R}\int_0^\infty|v-c|^2r^{\alpha-1}e^{-r^2}dr .
\end{equation}
\item[(ii)] For $a>-1$, $\beta>-4a-2$ and all $v$,
\begin{equation}\label{GP2a}
\int_0^\infty|v'|^2r^{\beta+2a+1}e^{-\frac{r^{2+2a}}{1+a}}dr\ \ge\ 4(1+a)\inf_{c\in\R}\int_0^\infty|v-c|^2r^{\beta+4a+1}e^{-\frac{r^{2+2a}}{1+a}}dr .
\end{equation}
\end{itemize}
In both cases the constant is sharp, with equality for $v=a_0+a_1r^2$ in \eqref{GP1} and for $v=a_0+a_1r^{2+2a}$ in \eqref{GP2a}. Moreover, both inequalities hold with the additional second-level term of Theorem \ref{TA}: with the same constant, $\inf_{c_0,c_1}\int|v-c_0-c_1r^2|^2$ resp.\ $\inf_{c_0,c_1}\int|v-c_0-c_1r^{2+2a}|^2$ (with the same weights) can be added to the right-hand side.
\end{corollary}

\begin{proof}
For (i) put $t=\sqrt2\,r$ in Theorem \ref{TA}. For (ii) put $t=\sqrt{2/(1+a)}\,r^{1+a}$, so that $e^{-t^2/2}=e^{-r^{2+2a}/(1+a)}$, $dt=\sqrt{2(1+a)}\,r^{a}dr$ and $\frac{dv}{dt}=\frac{v'(r)}{\sqrt{2(1+a)}\,r^{a}}$; Theorem \ref{TA} with $\alpha=\frac{\beta+4a+2}{1+a}>0$ then gives \eqref{GP2a}, the weights being $r^{(1+a)(\alpha-1)-a}=r^{\beta+2a+1}$ on the left and $r^{(1+a)(\alpha-1)+a}=r^{\beta+4a+1}$ on the right.
\end{proof}

\begin{remark}[The complete chain for Theorem \ref{TA}]\label{rem:GPchain}
The Laguerre expansion given after Theorem \ref{TA} shows that its two terms are the first two steps of a chain of arbitrary length, all of whose constants equal $2$. Denote by $\cP_i$ the space of the polynomials of degree at most $i$ in the variable $t^2$, i.e.\ the span of the Laguerre levels $j\le i$, and write $d\gamma_\alpha:=e^{-t^2/2}t^{\alpha-1}dt$. Then, for every $n\ge1$ and every $v$,
\begin{equation}\label{GPchain}
\int_0^\infty|v'|^2d\gamma_\alpha\ \ge\ 2\sum_{i=1}^n\ \inf_{q\in\cP_{i-1}}\int_0^\infty|v-q|^2d\gamma_\alpha ,
\end{equation}
all the constants being optimal in the sense that none of them can be increased while the others are kept fixed, with equality if and only if $v\in\cP_n$; the case $n=2$ is Theorem \ref{TA}. Moreover the chain is complete (see \cite{LL26} for this phenomenon for the Gaussian Poincar\'e inequality, and for its characterization through the arithmetic spectrum): letting $n\to\infty$ gives the identity
\begin{equation}\label{GPchaininfty}
\int_0^\infty|v'|^2d\gamma_\alpha\ =\ 2\sum_{i=1}^\infty\ \inf_{q\in\cP_{i-1}}\int_0^\infty|v-q|^2d\gamma_\alpha .
\end{equation}
Indeed, with $c_j:=|a_j|^2\frac{\Gamma(j+\frac\alpha2)}{j!}$ the squared components of the expansion above, the left-hand side of \eqref{GPchain} equals $2\sum_{j\ge1}j\,c_j$ and its $i$-th term equals $2\sum_{j\ge i}c_j$, so that \eqref{GPchain} amounts to $\sum_{j\ge1}j\,c_j\ge\sum_{j\ge1}\min\{j,n\}\,c_j$; this holds because $j\ge\min\{j,n\}$, with equality if and only if $c_j=0$ for $j>n$, and letting $n\to\infty$ gives \eqref{GPchaininfty} by monotone convergence. Taking $v=L_i^{(\frac\alpha2-1)}(\frac{t^2}{2})$, both sides of \eqref{GPchain} equal $2i\,c_i$ as soon as $n\ge i$, which shows that the $i$-th constant cannot be increased. The same statement holds for the two forms of Corollary \ref{cor:GP}, with the constants $4$ and $4(1+a)$ respectively, and for the radial Gaussian Poincar\'e inequality in $\R^N$, by the changes of variable used there.
\end{remark}

\subsection{Hamamoto's inequality: an extension and a sharp remainder term}\label{sec:ext}
The key point in the derivation of the best constant of \eqref{HUPsol} in \cite{Ham23} is the one-dimensional minimization problem
\begin{equation}\label{Hmin}
\min_{f\not\equiv0}\ \frac{\int_0^\infty|f''|^2x^{\mu+1}dx\,\int_0^\infty\big(x^2|f'|^2-\epsilon|f|^2\big)x^{\mu-1}dx}{\Big(\int_0^\infty|f'|^2x^{\mu}dx\Big)^2},
\end{equation}
where $\mu,\epsilon$ are two positive parameters with $\epsilon<\mu^2/4$. All functions in this subsection are real-valued; since the operators below have real coefficients and all the functionals are quadratic, the results extend to complex-valued functions by applying them to the real and imaginary parts. Throughout the paper we write, for such $\mu,\epsilon$,
\begin{equation}\label{QP}
Q[f]:=\int_0^\infty|f''|^2x^{\mu+1}dx,\quad
P_1[f]:=\int_0^\infty\big(x^2|f'|^2-\epsilon|f|^2\big)x^{\mu-1}dx,\quad
P_0[f]:=\int_0^\infty|f'|^2x^{\mu}dx ,
\end{equation}
and
\begin{equation}\label{bf0}
b:=\frac{\mu-\sqrt{\mu^2-4\epsilon}}{2},\qquad f_0(x):=e^{-x}{}_1F_1(b,\mu,x)={}_1F_1(\mu-b,\mu,-x),
\end{equation}
the last equality being Kummer's transformation. Hamamoto \cite[Theorem 5.2]{Ham23} proved that the minimum in \eqref{Hmin} is $\frac{c^2}{4}$ with $c=1+\sqrt{\mu^2-4\epsilon}$, attained exactly at the dilates $f_0(\lambda\,\cdot)$; a much shorter proof was given in \cite{Ham24}. In this subsection we first extend this inequality by a two-parameter family, and then we establish its sharp stability. Both rest on writing the deficit as the integral of a square. Put
\[
J[f]:=\int_0^\infty|f'|^2x^{\mu-1}dx,\qquad M_-[f]:=\int_0^\infty|f|^2x^{\mu-2}dx .
\]
The natural function space for these functionals is the completion $\cD_\mu$ of $C_c^\infty([0,\infty))$, the smooth functions with compact support in $[0,\infty)$ (not required to vanish at $0$), with respect to the norm
\begin{equation}\label{Dmu}
\|f\|_{\cD_\mu}^2:=\int_0^\infty|f''|^2x^{\mu+1}dx+\int_0^\infty|f'|^2x^{\mu+1}dx+\int_0^\infty|f'|^2x^{\mu-1}dx+\int_0^\infty|f|^2x^{\mu-1}dx ,
\end{equation}
which is the space $\bar{\cE}_\mu$ of \cite{HOT}. On $\cD_\mu$ the functionals $Q$, $P_1$, $J$ are finite, $P_0$ is finite by the Cauchy--Schwarz inequality between the two $|f'|^2$-terms of \eqref{Dmu}, and all the boundary terms in the integrations by parts below vanish, by density and \cite[Lemma 1]{HOT}. The extremals $f_0$ and $f_1$ of Theorem \ref{thm:1d} belong to $\cD_\mu$, since ${}_1F_1(\mu-b,\mu,-x)=O(x^{-(\mu-b)})$ as $x\to\infty$ and $\int^\infty x^{\mu-1-2(\mu-b)}dx<\infty$. The functional $M_-$, and the vanishing of the boundary term $|f|^2x^{\mu-1}$ at $x=0$, require in addition $\mu>1$, which we assume in the next statement; then $M_-$ is finite on $\cD_\mu$ by Hardy's inequality $\int_0^\infty|f|^2x^{\mu-2}dx\le\frac{4}{(\mu-1)^2}\int_0^\infty|f'|^2x^{\mu}dx$. All the values of $\mu$ occurring in this paper satisfy $\mu>1$.

\begin{proposition}\label{prop:ext}
Let $\mu>1$, $0<\epsilon<\frac{\mu^2}{4}$ and $0\le\eta\le\frac{\mu^2}{4}$, and set
\[
\alpha:=\frac{\mu+\sqrt{\mu^2-4\eta}}{2},\qquad\beta:=\frac{\mu+\sqrt{\mu^2-4\epsilon}}{2}=\mu-b .
\]
Then for all $f\in\cD_\mu$,
\begin{align}\label{ext}
\int_0^\infty|f''|^2x^{\mu+1}dx&-\eta\int_0^\infty|f'|^2x^{\mu-1}dx+\int_0^\infty\big(x^2|f'|^2-\epsilon|f|^2\big)x^{\mu-1}dx\notag\\
&\ge\ \Big(1+\mu+\sqrt{\mu^2-4\epsilon}-\sqrt{\mu^2-4\eta}\Big)\int_0^\infty|f'|^2x^{\mu}dx\notag\\
&\qquad-\frac{\mu+\sqrt{\mu^2-4\epsilon}}{2}\,\frac{\mu-\sqrt{\mu^2-4\eta}}{2}\,(\mu-1)\int_0^\infty|f|^2x^{\mu-2}dx ,
\end{align}
with equality if and only if $f=C\,e^{-x}{}_1F_1(\alpha-\beta,\alpha,x)$. More precisely, one has the identity
\begin{align}\label{extid}
&\int_0^\infty\Big|xf''+(\alpha+x)f'+\beta f\Big|^2x^{\mu-1}dx\notag\\
&\qquad=Q[f]-\eta J[f]+P_1[f]-\Big(1+\mu+\sqrt{\mu^2-4\epsilon}-\sqrt{\mu^2-4\eta}\Big)P_0[f]\notag\\
&\qquad\qquad+\beta(\mu-1)(\mu-\alpha)M_-[f].
\end{align}
For $\eta=0$ one has $\alpha=\mu$, the last term vanishes, and \eqref{extid} is the identity
\begin{equation}\label{RemainderIdt}
Q[f]+P_1[f]-\Big(1+\sqrt{\mu^2-4\epsilon}\Big)P_0[f]=\int_0^\infty\Big|xf''+(x+\mu)f'+(\mu-b)f\Big|^2x^{\mu-1}dx ,
\end{equation}
which holds for all $\mu>0$ and contains Hamamoto's inequality with its equality case $f=Cf_0$.
\end{proposition}

\begin{proof}
Expanding the square and integrating by parts (the boundary terms vanish on $\cD_\mu$, see \cite[Lemma 1]{HOT}), we get
\begin{align*}
&\int_0^\infty\Big|xf''+(\alpha+x)f'+\beta f\Big|^2x^{\mu-1}dx\\
&=\int_0^\infty|f''|^2x^{\mu+1}dx+\alpha^2\int_0^\infty|f'|^2x^{\mu-1}dx+\int_0^\infty|f'|^2x^{\mu+1}dx+\beta^2\int_0^\infty|f|^2x^{\mu-1}dx\\
&\quad+2\alpha\int_0^\infty f''f'x^{\mu}dx+2\int_0^\infty f''f'x^{\mu+1}dx+2\beta\int_0^\infty f''f\,x^{\mu}dx\\
&\quad+2\alpha\int_0^\infty|f'|^2x^{\mu}dx+2\alpha\beta\int_0^\infty f'f\,x^{\mu-1}dx+2\beta\int_0^\infty f'f\,x^{\mu}dx\\
&=\int_0^\infty|f''|^2x^{\mu+1}dx+(\alpha^2-\alpha\mu)\int_0^\infty|f'|^2x^{\mu-1}dx+\int_0^\infty|f'|^2x^{\mu+1}dx\\
&\quad+(\beta^2-\beta\mu)\int_0^\infty|f|^2x^{\mu-1}dx-\big(\mu+1-2\alpha+2\beta\big)\int_0^\infty|f'|^2x^{\mu}dx\\
&\quad-\beta(\mu-1)(\alpha-\mu)\int_0^\infty|f|^2x^{\mu-2}dx .
\end{align*}
where the six cross terms were integrated by parts as follows:
\begin{align*}
2\alpha\int_0^\infty f''f'x^{\mu}dx&=\alpha\int_0^\infty(|f'|^2)'x^{\mu}dx=-\alpha\mu\int_0^\infty|f'|^2x^{\mu-1}dx,\\
2\int_0^\infty f''f'x^{\mu+1}dx&=-(\mu+1)\int_0^\infty|f'|^2x^{\mu}dx,\\
2\beta\int_0^\infty f''f\,x^{\mu}dx&=-2\beta\int_0^\infty|f'|^2x^{\mu}dx-2\beta\mu\int_0^\infty f'f\,x^{\mu-1}dx\\
&=-2\beta\int_0^\infty|f'|^2x^\mu dx+\beta\mu(\mu-1)\int_0^\infty|f|^2x^{\mu-2}dx,\\
2\alpha\beta\int_0^\infty f'f\,x^{\mu-1}dx&=-\alpha\beta(\mu-1)\int_0^\infty|f|^2x^{\mu-2}dx,\\
2\beta\int_0^\infty f'f\,x^{\mu}dx&=-\beta\mu\int_0^\infty|f|^2x^{\mu-1}dx ,
\end{align*}
while $2\alpha\int|f'|^2x^\mu$ is kept as it is. Choosing $\alpha$ and $\beta$ such that $\alpha^2-\alpha\mu=-\eta$ and $\beta^2-\beta\mu=-\epsilon$, i.e.\ as in the statement, and observing that $\mu+1-2\alpha+2\beta=1+\mu+\sqrt{\mu^2-4\epsilon}-\sqrt{\mu^2-4\eta}$, we obtain \eqref{extid}, and \eqref{ext} follows since the left-hand side is nonnegative. Equality holds if and only if $xf''+(\alpha+x)f'+\beta f=0$. Let $f(x)=e^{-x}\psi(x)$; then $f'=e^{-x}(\psi'-\psi)$, $f''=e^{-x}(\psi''-2\psi'+\psi)$, and $\psi$ solves Kummer's equation
\[
x\psi''+(\alpha-x)\psi'-(\alpha-\beta)\psi=0 ,
\]
whose solutions giving $f\in\cD_\mu$ are the multiples of ${}_1F_1(\alpha-\beta,\alpha,x)$ (the second solution behaves like $x^{1-\alpha}$ near $0$, and then $J[f]=\infty$); the corresponding $f$ belong to $\cD_\mu$, since $f(x)={}_1F_1(\beta,\alpha,-x)=O(x^{-\beta})$ at infinity. For $\eta=0$ no term with $x^{\mu-2}$ is produced, so the computation is valid for $\mu>0$; the equation $xf''+(x+\mu)f'+(\mu-b)f=0$ then gives $\psi={}_1F_1(b,\mu,x)$, i.e.\ $f=Cf_0$.
\end{proof}

It is worth noting that the coefficient of $\int|f|^2x^{\mu-2}$ in \eqref{ext} is nonpositive, so that \eqref{ext} does not reduce to a two-term inequality when $\eta>0$. Such perturbations of Hamamoto's functional arise, for instance, in the weighted uncertainty principles for solenoidal fields away from the Heisenberg point.

Let us explain how the operator $T$ of \eqref{RemainderIdt} is found, which is at the same time another way of organizing the computation of Proposition \ref{prop:ext}. One looks for the deficit of \eqref{Hmin} in sum form,
\[
R:=\int_0^\infty|f''|^2x^{\mu+1}dx+\int_0^\infty\big(x^2|f'|^2-\epsilon|f|^2\big)x^{\mu-1}dx-M\int_0^\infty|f'|^2x^{\mu}dx ,
\]
i.e.\ $R=\int_0^\infty x^{\mu-1}\big[x^2|f''|^2+x^2|f'|^2-\epsilon|f|^2-Mx|f'|^2\big]dx$,
with $M$ to be determined, in the form of the integral of a square, $R=\int_0^\infty x^{\mu-1}\big|xf''+Axf'+Bf+Cf'\big|^2dx$, with constants $A,B,C$. Expanding as in the proof of Proposition \ref{prop:ext} (with $A$, $C$, $B$ in place of $1$, $\alpha$, $\beta$),
\begin{align*}
&\int_0^\infty x^{\mu-1}\big|xf''+Axf'+Bf+Cf'\big|^2dx\\
&\quad=\int_0^\infty x^{\mu+1}|f''|^2dx+A^2\int_0^\infty x^{\mu+1}|f'|^2dx\\
&\qquad+(B^2-AB\mu)\int_0^\infty x^{\mu-1}|f|^2dx+(C^2-C\mu)\int_0^\infty x^{\mu-1}|f'|^2dx\\
&\qquad+\big(2AC-A(\mu+1)-2B\big)\int_0^\infty x^{\mu}|f'|^2dx\\
&\qquad+B(\mu-1)(\mu-C)\int_0^\infty x^{\mu-2}|f|^2dx ,
\end{align*}
so that one needs
\begin{gather*}
A^2=1,\qquad B^2-AB\mu=-\epsilon,\qquad C^2-C\mu=0,\\
-A(\mu+1)-2B+2AC=-M,\qquad B(\mu-1)(\mu-C)=0 .
\end{gather*}
For $\mu\ne1$ the third and fifth equations force $C=\mu$ (the other root $C=0$ of the third equation is excluded by the fifth, as $B\ne0$); then $A=1$, $B=\frac{\mu+\sqrt{\mu^2-4\epsilon}}{2}=\mu-b$ from the second equation (the other root $\frac{\mu-\sqrt{\mu^2-4\epsilon}}2=b$ would also do, giving the identity with $b$ and $\mu-b$ exchanged), and
\[
M=A(\mu+1)+2B-2AC=\mu+1+\mu+\sqrt{\mu^2-4\epsilon}-2\mu=1+\sqrt{\mu^2-4\epsilon} ,
\]
which is \eqref{RemainderIdt}. This also shows in what sense \eqref{RemainderIdt} is rigid: the sum-form deficit at the level $M$ is an exact square of a second order operator with constant coefficients (in the sense above) only for $M=1+\sqrt{\mu^2-4\epsilon}$.

The identity \eqref{RemainderIdt} contains the product form of Hamamoto's inequality, i.e.\ \cite[Theorem 1]{Ham24}, with an explicit remainder. Indeed, applying \eqref{RemainderIdt} to $f_\lambda(x):=f(\lambda x)$, $\lambda>0$, and using $\int|f_\lambda''|^2x^{\mu+1}=\lambda^{2-\mu}\int|f''|^2x^{\mu+1}$, $\int(x^2|f_\lambda'|^2-\epsilon|f_\lambda|^2)x^{\mu-1}=\lambda^{-\mu}\int(x^2|f'|^2-\epsilon|f|^2)x^{\mu-1}$, $\int|f_\lambda'|^2x^\mu=\lambda^{1-\mu}\int|f'|^2x^\mu$ and, on the right-hand side, $xf_\lambda''(x)+(x+\mu)f_\lambda'(x)+(\mu-b)f_\lambda(x)=\lambda(\lambda x)f''(\lambda x)+(\lambda x+\mu\lambda)f'(\lambda x)+(\mu-b)f(\lambda x)$, so that $\int x^{\mu-1}|Tf_\lambda|^2dx=\lambda^{-\mu}\int x^{\mu-1}|\lambda xf''+(x+\mu\lambda)f'+(\mu-b)f|^2dx$ after the change of variable $\lambda x\mapsto x$, we get, after multiplication by $\lambda^{\mu-1}$,
\begin{align*}
&\lambda\int_0^\infty|f''|^2x^{\mu+1}dx+\frac1\lambda\int_0^\infty\big(x^2|f'|^2-\epsilon|f|^2\big)x^{\mu-1}dx-\Big(1+\sqrt{\mu^2-4\epsilon}\Big)\int_0^\infty|f'|^2x^{\mu}dx\\
&\qquad=\frac1\lambda\int_0^\infty x^{\mu-1}\Big|\lambda xf''+(x+\mu\lambda)f'+(\mu-b)f\Big|^2dx .
\end{align*}
Choosing
\[
\lambda=\Bigg(\frac{\int_0^\infty\big(x^2|f'|^2-\epsilon|f|^2\big)x^{\mu-1}dx}{\int_0^\infty|f''|^2x^{\mu+1}dx}\Bigg)^{\frac12},
\]
for which the first two terms are equal, we obtain the exact remainder identity for the product form,
\begin{align}\label{prodid1d}
&\Big(\int_0^\infty|f''|^2x^{\mu+1}dx\Big)^{\frac12}\Big(\int_0^\infty\big(x^2|f'|^2-\epsilon|f|^2\big)x^{\mu-1}dx\Big)^{\frac12}-\frac{1+\sqrt{\mu^2-4\epsilon}}{2}\int_0^\infty|f'|^2x^{\mu}dx\notag\\
&\qquad=\frac{1}{2\lambda}\int_0^\infty x^{\mu-1}\Big|\lambda xf''+(x+\mu\lambda)f'+(\mu-b)f\Big|^2dx\ \ge0 ,
\end{align}
which is the statement that the minimum in \eqref{Hmin} is $\frac{(1+\sqrt{\mu^2-4\epsilon})^2}{4}$, with equality if and only if $\lambda xf''+(x+\mu\lambda)f'+(\mu-b)f=0$, i.e.\ $f=Cf_0(\cdot/\lambda)$. This is exactly the one-parameter family of squares of \cite{Ham24, HOT}, and it is the one-dimensional counterpart of the passage from the sum form to the product form of Section \ref{sec:prod}.

We now turn to the stability of \eqref{RemainderIdt}. For simplicity, we denote
\begin{equation}\label{TL}
Tf:=xf''+(x+\mu)f'+(\mu-b)f,\qquad Lf:=-xf''-\mu f',\qquad d\mu:=x^{\mu-1}dx ,
\end{equation}
so that \eqref{RemainderIdt} reads $Q[f]+P_1[f]-(1+\sqrt{\mu^2-4\epsilon})P_0[f]=\int_0^\infty|Tf|^2d\mu$, and
\[
\int_0^\infty fLf\,d\mu=\int_0^\infty|f'|^2x\,d\mu=P_0[f],\qquad
\int_0^\infty fLf_0\,d\mu=\int_0^\infty f'f_0'\,x\,d\mu .
\]
We have $Tf=0$ if and only if $f=Cf_0$. The following theorem, which is the one-dimensional result behind both groups of our main results, gives the stability estimate for the identity \eqref{RemainderIdt} with the sharp constant.

\begin{theorem}\label{thm:1d}
Let $\mu>0$ and $0<\epsilon<\frac{\mu^2}4$, and let $T$, $L$ be the operators in \eqref{TL}. Then
\begin{equation}\label{stab1d}
\int_0^\infty|Tf|^2d\mu\ \ge\ 2\int_0^\infty fLf\,d\mu
\end{equation}
for all $f\in\cD_\mu$ satisfying $\int_0^\infty fLf_0\,d\mu=0$. Moreover, the constant $2$ is optimal and is achieved at
\begin{align}\label{f1}
f_1(x)&:=(\mu-b)\,{}_1F_1(\mu-b+1,\mu,-x)-\frac{\mu-2b+1}{2}\,{}_1F_1(\mu-b,\mu,-x)\notag\\
&=-\frac12\Big[(1-\mu)f_0(x)-2xf_0'(x)\Big],
\end{align}
up to a constant factor. Consequently, for all $f\in\cD_\mu$,
\begin{align}\label{stab1dinf}
Q[f]+P_1[f]-\Big(1+\sqrt{\mu^2-4\epsilon}\Big)P_0[f]&=\int_0^\infty|Tf|^2d\mu\ \ge\ 2\inf_{C\in\R}\int_0^\infty|f'-Cf_0'|^2x\,d\mu\notag\\
&=2\Big[\int_0^\infty|f'|^2x\,d\mu-\frac{\big(\int_0^\infty f'f_0'x\,d\mu\big)^2}{\int_0^\infty|f_0'|^2x\,d\mu}\Big],
\end{align}
with equality exactly for $f$ in the linear span of $f_0$ and $f_1$.
\end{theorem}

\begin{proof}
Since $Tf_0=0$, \eqref{stab1dinf} follows from \eqref{stab1d} applied to $f-Cf_0$ with $C$ realizing the infimum, i.e.\ $C=\int f'f_0'x\,d\mu/\int|f_0'|^2x\,d\mu$, for which $\int(f-Cf_0)Lf_0\,d\mu=0$. We prove \eqref{stab1d}. The idea is to apply the Hankel transform to reduce the order of derivatives in the operators $T$ and $L$.

\emph{Step 1.} Setting $x=s^2/4$ and defining $F(s)=f(s^2/4)=f(x)$, we get $f'(x)=\frac2sF'(s)$ and $f''(x)=\frac4{s^2}F''(s)-\frac4{s^3}F'(s)$. Then the operator $T$ can be rewritten as
\[
Tf=F''(s)+\Big(\frac{2\mu-1}{s}+\frac s2\Big)F'(s)+(\mu-b)F(s),
\]
which implies that
\[
\int_0^\infty|Tf|^2d\mu=\frac{1}{2^{2\mu-1}}\int_0^\infty\Big|F''(s)+\Big(\frac{2\mu-1}{s}+\frac s2\Big)F'(s)+(\mu-b)F(s)\Big|^2s^{2\mu-1}ds ,
\]
and likewise
\[
\int_0^\infty fLf\,d\mu=\int_0^\infty|f'|^2x\,d\mu=\frac{1}{2^{2\mu-1}}\int_0^\infty|F'|^2s^{2\mu-1}ds .
\]

\emph{Step 2: the Hankel transform.} We recall some properties of the Hankel transform that are useful in our arguments (see \cite[Chapter XIV]{Watson}). The Hankel transform of order $\nu>-1$ of a function $\phi$, given by
\[
H_\nu\phi(\rho):=\int_0^\infty\phi(s)J_\nu(\rho s)\,s\,ds ,
\]
satisfies $\int_0^\infty H_\nu\phi\,H_\nu\psi\,\rho\,d\rho=\int_0^\infty\phi\psi\,s\,ds$ (Parseval's identity) and $H_\nu^{-1}=H_\nu$ (Hankel's inversion theorem), first on $C_c^\infty(0,\infty)$ and then by density on $L^2((0,\infty),s\,ds)$. We shall also use the following two elementary rules. First, since the Bessel operator $\phi\mapsto\phi''+\frac1s\phi'-\frac{\nu^2}{s^2}\phi=\frac1s(s\phi')'-\frac{\nu^2}{s^2}\phi$ is symmetric with respect to the measure $s\,ds$, and since $J_\nu(\rho s)$ satisfies the Bessel equation $J_\nu''(t)+\frac1tJ_\nu'(t)+(1-\frac{\nu^2}{t^2})J_\nu(t)=0$ with $t=\rho s$, two integrations by parts give
\[
H_\nu\Big(\phi''+\frac1s\phi'-\frac{\nu^2}{s^2}\phi\Big)(\rho)=\int_0^\infty\phi(s)\Big(\frac{d^2}{ds^2}+\frac1s\frac{d}{ds}-\frac{\nu^2}{s^2}\Big)J_\nu(\rho s)\,s\,ds=-\rho^2H_\nu\phi(\rho),
\]
the boundary terms vanishing for $\phi\in C_c^\infty(0,\infty)$. Second, since $H_\nu$ intertwines the dilations, $H_\nu\big(\phi(\lambda\,\cdot)\big)(\rho)=\lambda^{-2}H_\nu\phi(\rho/\lambda)$, differentiating this identity with respect to $\lambda$ at $\lambda=1$ gives
\[
H_\nu(s\phi')(\rho)=-2H_\nu\phi(\rho)-\rho\frac{d}{d\rho}H_\nu\phi(\rho).
\] By Parseval's identity applied to $s^\nu F$,
\[
\int_0^\infty|F(s)|^2s^{2\nu+1}ds=\int_0^\infty|v(\rho)|^2\rho^{2\nu+1}d\rho,\qquad
v(\rho):=\rho^{-\nu}\int_0^\infty F(s)J_\nu(s\rho)s^{\nu+1}ds ,
\]
which we call the extended Hankel transform of $F$, and likewise for the bilinear form. Denote $n=2\mu=2\nu+2$, i.e.\ $\nu=\mu-1$, and $\beta=\mu-b$. Set $g=s^\nu F$; then
\[
s^\nu D_nF:=s^\nu\Big(F''+\frac{n-1}{s}F'\Big)=g''+\frac1sg'-\frac{\nu^2}{s^2}g=:B_\nu g ,
\]
where $B_\nu$ is the Bessel operator; the identity $s^\nu D_nF=B_\nu(s^\nu F)$ follows from $(s^\nu F)'=s^\nu F'+\nu s^{\nu-1}F$ and $(s^\nu F)''=s^\nu F''+2\nu s^{\nu-1}F'+\nu(\nu-1)s^{\nu-2}F$, the terms in $F$ cancelling because $\nu(\nu-1)+\nu-\nu^2=0$ and those in $F'$ giving $(2\nu+1)s^{\nu-1}F'=(n-1)s^{\nu-1}F'$. By the first rule above, $H_\nu(B_\nu g)=-\rho^2H_\nu g$, i.e.\ $H_\nu(s^\nu D_nF)=-\rho^{2+\nu}v$: the extended Hankel transform of $D_nF$ is $-\rho^2v$. Moreover $s^\nu(sF')=sg'-\nu g$ and $H_\nu g=\rho^\nu v$, and by the second rule,
\[
H_\nu\big(s^\nu(sF')\big)=H_\nu(sg'-\nu g)=-2\rho^\nu v-\rho(\rho^\nu v)'-\nu\rho^\nu v=\rho^\nu(-nv-\rho v') ,
\]
so the extended Hankel transform of $sF'$ is $-nv-\rho v'$. This allows us to compute the extended Hankel transform of $Tf=D_nF+\frac s2F'+\beta F$, which is
\[
-\rho^2v-\frac n2v-\frac\rho2v'+\beta v=-\Big(\frac\rho2v'+\rho^2v+bv\Big),
\]
and therefore
\begin{equation}\label{TH}
\int_0^\infty|Tf|^2d\mu=\frac{1}{2^{2\mu-1}}\int_0^\infty\Big|\frac\rho2v'+\rho^2v+bv\Big|^2\rho^{2\mu-1}d\rho .
\end{equation}
Since
\begin{align*}
\int_0^\infty|F'|^2s^{2\mu-1}ds&=-\int_0^\infty F\,(s^{n-1}F')'ds=-\int_0^\infty(s^\nu F)(s^\nu D_nF)\,s\,ds\\
&=-\int_0^\infty H_\nu(s^\nu F)H_\nu(s^\nu D_nF)\,\rho\,d\rho=\int_0^\infty|v|^2\rho^{2\mu+1}d\rho ,
\end{align*}
we obtain
\begin{equation}\label{LH}
\int_0^\infty fLf\,d\mu=\frac{1}{2^{2\mu-1}}\int_0^\infty|v|^2\rho^{2\mu+1}d\rho ,
\qquad
\int_0^\infty fLf_0\,d\mu=\frac{1}{2^{2\mu-1}}\int_0^\infty v\,v_0\,\rho^{2\mu+1}d\rho ,
\end{equation}
where $v_0$ is the extended Hankel transform of $F_0(s)=f_0(s^2/4)$. Since $Tf_0=0$, \eqref{TH} gives $\frac\rho2v_0'+\rho^2v_0+bv_0=0$, i.e.\ $v_0=C\rho^{-2b}e^{-\rho^2}$.

\emph{Step 3: reduction to a Laguerre expansion.} By \eqref{TH} and \eqref{LH}, it suffices to establish
\begin{equation}\label{stabEst}
\int_0^\infty\Big|\frac\rho2v'+\rho^2v+bv\Big|^2\rho^{2\mu-1}d\rho\ \ge\ K\int_0^\infty|v|^2\rho^{2\mu+1}d\rho,
\qquad\text{where}\quad\int_0^\infty v\,\rho^{2\mu-2b+1}e^{-\rho^2}d\rho=0 ,
\end{equation}
with $K=2$. Setting $t=\rho^2$ and $w(t)=v(\sqrt t)$, so that $w_0(t)=Ct^{-b}e^{-t}$, the estimate \eqref{stabEst} becomes
\begin{equation}\label{stabEst2}
\int_0^\infty|tw'+(t+b)w|^2t^{\mu-1}dt\ \ge\ K\int_0^\infty|w|^2t^{\mu}dt .
\end{equation}
Since $tw_0'+(t+b)w_0=0$, writing $w=w_0h$ we have
\[
tw'+(t+b)w=\big(tw_0'+(t+b)w_0\big)h+tw_0h'=tw_0h' ,
\]
which implies that \eqref{stabEst2} is equivalent to
\[
\int_0^\infty|h'|^2t^{\mu-2b+1}e^{-2t}dt\ \ge\ K\int_0^\infty|h|^2t^{\mu-2b}e^{-2t}dt
\]
for all $h$ such that $\int_0^\infty h\,t^{\mu-2b}e^{-2t}dt=0$. With $z=2t$ and $g(z)=h(z/2)$, the above estimate turns into
\begin{equation}\label{stabEst3}
\int_0^\infty|g'|^2z^{\mu-2b+1}e^{-z}dz\ \ge\ \frac K2\int_0^\infty|g|^2z^{\mu-2b}e^{-z}dz ,
\end{equation}
where $\int_0^\infty g\,z^{\mu-2b}e^{-z}dz=0$. With $\mu-2b=\frac\gamma2-1$, that is, $\frac\gamma2-1=\sqrt{\mu^2-4\epsilon}$, we decompose $g$ in generalized Laguerre polynomials,
\[
g(z)=\sum_{k=0}^\infty a_kL_k^{\frac\gamma2-1}(z) .
\]
Hence $g'(z)=-\sum_{k=1}^\infty a_kL_{k-1}^{\gamma/2}(z)$, which by the orthogonality relations of the Laguerre polynomials \cite[Chapter V]{Szego} gives
\[
\int_0^\infty|g'|^2z^{\gamma/2}e^{-z}dz=\sum_{k=1}^\infty k|a_k|^2\frac{\Gamma(k+\frac\gamma2)}{k!},\qquad
\int_0^\infty|g|^2z^{\gamma/2-1}e^{-z}dz=\sum_{k=0}^\infty|a_k|^2\frac{\Gamma(k+\frac\gamma2)}{k!} .
\]
Since $\int_0^\infty g\,z^{\mu-2b}e^{-z}dz=0$, we have $a_0=0$. This implies that \eqref{stabEst3} holds with $K=2$, and it is sharp with the equality case $g=a_1L_1^{\frac\gamma2-1}$; it also shows that the values of the quotient in \eqref{stabEst3} at its critical points are the numbers $k=1,2,\dots$, attained at $L_k^{\frac\gamma2-1}$, which are mutually orthogonal.

\emph{Step 4: the extremal.} In the equality case, using $L_1^{\frac\gamma2-1}(z)=\frac\gamma2-z$, we get
\[
v(\rho)=a_1w_0(\rho^2)L_1^{\frac\gamma2-1}(2\rho^2)=C\rho^{-2b}e^{-\rho^2}\Big(\frac\gamma2-2\rho^2\Big)=Ce^{-\rho^2}\big(\rho^{2-2b}-A\rho^{-2b}\big),
\]
where $A=\frac{\mu-2b+1}{2}$.
Since the Hankel transform is self-reciprocal, we have
\begin{align*}
F(s)&=s^{1-\mu}\int_0^\infty v(\rho)J_{\mu-1}(s\rho)\rho^{\mu}d\rho\\
&=Cs^{1-\mu}\Big(\int_0^\infty\rho^{\mu+2-2b}e^{-\rho^2}J_{\mu-1}(s\rho)d\rho-A\int_0^\infty\rho^{\mu-2b}e^{-\rho^2}J_{\mu-1}(s\rho)d\rho\Big).
\end{align*}
Applying the formula \cite[13.3]{Watson}, valid for $\alpha>0$ and $m+\nu>-1$,
\[
\int_0^\infty x^{m}e^{-\alpha x^2}J_\nu(\beta x)dx=\frac{\beta^\nu\,\Gamma\big(\frac{\nu+m+1}{2}\big)}{2^{\nu+1}\alpha^{\frac{\nu+m+1}{2}}\Gamma(\nu+1)}\,{}_1F_1\Big(\frac{\nu+m+1}{2};\nu+1;-\frac{\beta^2}{4\alpha}\Big)
\]
we get
\begin{align*}
\int_0^\infty\rho^{\mu+2-2b}e^{-\rho^2}J_{\mu-1}(s\rho)d\rho&=\frac{s^{\mu-1}\Gamma(\mu-b+1)}{2^\mu\Gamma(\mu)}{}_1F_1\Big(\mu-b+1,\mu,-\frac{s^2}4\Big),\\
\int_0^\infty\rho^{\mu-2b}e^{-\rho^2}J_{\mu-1}(s\rho)d\rho&=\frac{s^{\mu-1}\Gamma(\mu-b)}{2^\mu\Gamma(\mu)}{}_1F_1\Big(\mu-b,\mu,-\frac{s^2}4\Big).
\end{align*}
Then
\[
F(s)=\frac{C\Gamma(\mu-b)}{2^\mu\Gamma(\mu)}\Big[(\mu-b)\,{}_1F_1\Big(\mu-b+1,\mu,-\frac{s^2}4\Big)-A\,{}_1F_1\Big(\mu-b,\mu,-\frac{s^2}4\Big)\Big],
\]
which implies that, up to a constant factor, $f=f_1$ in the first form of \eqref{f1}. For the second form, put $a_1:=\mu-b$ and $z:=-x$, so that $f_0(x)={}_1F_1(a_1,\mu,z)$ and, by the derivative formula $\frac{d}{dz}{}_1F_1(a,c,z)=\frac ac\,{}_1F_1(a+1,c+1,z)$, $f_0'(x)=-\frac{a_1}{\mu}{}_1F_1(a_1+1,\mu+1,z)$. The contiguous relation $z\,{}_1F_1(a+1,c+1,z)=c\big[{}_1F_1(a+1,c,z)-{}_1F_1(a,c,z)\big]$ \cite[Chapter 13]{AS} then gives
\[
-2xf_0'(x)=-\frac{2a_1}{\mu}\,z\,{}_1F_1(a_1+1,\mu+1,z)=-2a_1\Big[{}_1F_1(a_1+1,\mu,z)-{}_1F_1(a_1,\mu,z)\Big],
\]
hence
\[
-\frac12\Big[(1-\mu)f_0-2xf_0'\Big]=a_1\,{}_1F_1(a_1+1,\mu,z)-\frac{1-\mu+2a_1}{2}\,{}_1F_1(a_1,\mu,z),
\]
and $1-\mu+2a_1=\mu-2b+1$, which is the first form of \eqref{f1}. The same formula with $m=\mu-2b$ alone shows that the extended Hankel transform of $F_0$ is $w_0(\rho^2)$, as used above.
\end{proof}

For later use we record the effect of dilations. For $\tau>0$ and $f_\tau(x):=f(x/\tau)$ one has, by the change of variable $x=\tau y$,
\begin{equation}\label{scaling}
Q[f_\tau]=\tau^{\mu-2}Q[f],\qquad P_1[f_\tau]=\tau^{\mu}P_1[f],\qquad P_0[f_\tau]=\tau^{\mu-1}P_0[f] .
\end{equation}
Consequently, for every $\lambda>0$, applying Theorem \ref{thm:1d} to $f_{1/\lambda}=f(\lambda\,\cdot)$ and using $f_{1/\lambda}-Cf_0=(f-Cf_0(\cdot/\lambda))_{1/\lambda}$,
\begin{align*}
&\lambda^2Q[f]+P_1[f]-\lambda\Big(1+\sqrt{\mu^2-4\epsilon}\Big)P_0[f]\\
&\qquad=\lambda^{\mu}\Big(Q[f_{1/\lambda}]+P_1[f_{1/\lambda}]-\Big(1+\sqrt{\mu^2-4\epsilon}\Big)P_0[f_{1/\lambda}]\Big)\\
&\qquad\ge\ 2\lambda^{\mu}\inf_{C}P_0\big[f_{1/\lambda}-Cf_0\big]=2\lambda\inf_CP_0\big[f-Cf_0(\cdot/\lambda)\big],
\end{align*}
i.e.\ Theorem \ref{thm:1d} holds at every scale, with the extremal $f_0(\cdot/\lambda)$ and the remainder $2\lambda\inf_CP_0[f-Cf_0(\cdot/\lambda)]$ for the functional $\lambda^2Q+P_1$; in Section \ref{sec:sol} we meet the case $\lambda=2$, for which the functional is $4Q+P_1$, the extremal is $f_0(x/2)=e^{-x/2}{}_1F_1(b,\mu,x/2)$ and the remainder constant is $4$.

\begin{remark}[A shortcut, higher levels and the complete chain]\label{rem:chain}
Step 3 of the proof can be replaced by a direct reduction to the weighted Gaussian Poincar\'e inequality of Theorem \ref{TA}, which avoids the Laguerre expansion altogether. Indeed, writing
\begin{equation}\label{gauge}
v(\rho)=\eta(\rho)\,e^{-\rho^2}\rho^{-2b} ,
\end{equation}
that is, $v=\eta\,v_0$ with the function $v_0$ of Step 2, we have $v'=\eta'e^{-\rho^2}\rho^{-2b}-2\rho v-\frac{2b}{\rho}v$, so that the three terms of \eqref{stabEst} combine into
\begin{equation}\label{gaugeid}
\frac\rho2v'(\rho)+\rho^2v(\rho)+bv(\rho)=\frac\rho2\,\eta'(\rho)\,e^{-\rho^2}\rho^{-2b} ,
\end{equation}
the terms $\rho^2v$ and $bv$ cancelling exactly against the two terms produced by the derivative of $e^{-\rho^2}\rho^{-2b}$. Hence \eqref{stabEst} is equivalent to
\[
\frac14\int_0^\infty|\eta'|^2e^{-2\rho^2}\rho^{2\mu-4b+1}d\rho\ \ge\ K\int_0^\infty|\eta|^2e^{-2\rho^2}\rho^{2\mu-4b+1}d\rho ,
\]
under the constraint $\int_0^\infty\eta\,e^{-2\rho^2}\rho^{2\mu-4b+1}d\rho=0$, both sides carrying the same weight. Putting $s=2\rho$ and $w(s)=\eta(s/2)$, so that $\eta'(\rho)=2w'(s)$ and $e^{-2\rho^2}=e^{-s^2/2}$, this becomes
\[
\int_0^\infty|w'|^2e^{-\frac{s^2}{2}}s^{\gamma-1}ds\ \ge\ K\int_0^\infty|w|^2e^{-\frac{s^2}{2}}s^{\gamma-1}ds,\qquad\int_0^\infty w\,e^{-\frac{s^2}{2}}s^{\gamma-1}ds=0 ,
\]
with $\gamma:=2\mu-4b+2=2\omega+2>0$, where $\omega:=\sqrt{\mu^2-4\epsilon}=\mu-2b$; this is the same parameter $\gamma$ as in Step 3. The last display is exactly Theorem \ref{TA} for the parameter $\gamma$, which holds with $K=2$, the constant being sharp and attained by $w=a_0+a_1s^2$; under the orthogonality constraint, $w$ is a multiple of $L_1^{\frac\gamma2-1}(\frac{s^2}{2})$, in agreement with Step 4.

The proof shows more. The values of the quotient $\frac{Q[f]+P_1[f]}{P_0[f]}$ at its critical points are exactly the numbers $1+\omega+2k$, $k=0,1,2,\dots$, the value of the level $k$ being attained precisely at the multiples of
\begin{equation}\label{fk1d}
f_k(x):=\sum_{m=0}^k\binom{k+\omega}{k-m}\frac{(-2)^m}{m!}\,\Gamma(\mu-b+m)\ {}_1F_1(\mu-b+m,\mu,-x) ,
\end{equation}
the functions of different levels being orthogonal with respect to the bilinear form $(f,u)\mapsto\int fLu\,d\mu$; for $k=0$ and $k=1$ these are $f_0$ and $f_1$ up to normalization. Indeed, the level $k$ corresponds to $g=L_k^{\frac\gamma2-1}$ in Step 3, i.e.\ to $v=w_0(\rho^2)L_k^{\frac\gamma2-1}(2\rho^2)$, and \eqref{fk1d} is obtained by applying the formula of Step 4 to each monomial of $L_k^{\frac\gamma2-1}$.

Correspondingly, the estimate \eqref{stab1d} is the first step of a chain of arbitrary length, all of whose constants equal $2$: writing $E_i:=\operatorname{span}\{f_0,f_1,\dots,f_i\}$, one has, for every $n\ge1$ and every $f\in\cD_\mu$,
\begin{equation}\label{stab1dn}
\int_0^\infty|Tf|^2d\mu\ \ge\ 2\sum_{i=1}^{n}\ \inf_{u\in E_{i-1}}\int_0^\infty|f'-u'|^2x\,d\mu ,
\end{equation}
the $i$-th term being the squared distance from $f'$ to the derivatives of $E_{i-1}$; the case $n=1$ is \eqref{stab1dinf}. All the constants are optimal, in the sense that none of them can be increased while the others are kept fixed, the $i$-th one being achieved at $f=f_i$, and equality holds in \eqref{stab1dn} exactly for $f\in E_n$. Moreover the chain is complete: letting $n\to\infty$ gives the identity
\begin{equation}\label{stab1dinfty}
\int_0^\infty|Tf|^2d\mu\ =\ 2\sum_{i=1}^{\infty}\ \inf_{u\in E_{i-1}}\int_0^\infty|f'-u'|^2x\,d\mu ,\qquad f\in\cD_\mu .
\end{equation}
These statements are Remark \ref{rem:GPchain} read through the dictionary of Step 3: the space $E_{i-1}$ corresponds to $\cP_{i-1}$, i.e.\ to the Laguerre levels $j\le i-1$ of $g$, and the two members of \eqref{stab1dn} correspond to those of \eqref{GPchain}.
\end{remark}

It is worth noting that the Hankel transform explains in one line why the minimization problem \eqref{Hmin} can be solved explicitly: it turns the fourth order operator $T$ into a first order one, and the quotient into the Gaussian Poincar\'e inequality of Theorem \ref{TA}.

\subsection{Change of functions}
The following elementary identities will be used repeatedly. Let $m,\gamma,a,b\in\R$ and $f\in C_0^\infty(0,\infty)$, and write $f(r)=r^\gamma g(r)$. Since $f'=r^\gamma\big(g'+\frac\gamma rg\big)$ and $f''=r^\gamma\big(g''+\frac{2\gamma}rg'+\frac{\gamma(\gamma-1)}{r^2}g\big)$, one has
\[
f''+\frac{m-1}{r}f'-\gamma(m+\gamma-2)\frac{f}{r^2}=r^\gamma\Big(g''+\frac{m+2\gamma-1}{r}g'\Big),
\]
the coefficient of $g/r^2$ being $\gamma(\gamma-1)+\gamma(m-1)-\gamma(m+\gamma-2)=0$. Squaring and integrating against $r^{m-2a-1}$, the cross term is
\begin{align*}
&2(m+2\gamma-1)\int_0^\infty g''g'\,r^{m+2\gamma-2a-2}dr=(m+2\gamma-1)\int_0^\infty(|g'|^2)'r^{m+2\gamma-2a-2}dr\\
&\qquad=-(m+2\gamma-1)(m+2\gamma-2a-2)\int_0^\infty|g'|^2r^{m+2\gamma-2a-3}dr ,
\end{align*}
and $(m+2\gamma-1)^2-(m+2\gamma-1)(m+2\gamma-2a-2)=(1+2a)(m+2\gamma-1)$. Similarly $|f'|^2+\gamma(m+\gamma-2)\frac{|f|^2}{r^2}=r^{2\gamma}\big(|g'|^2+\frac{2\gamma}rgg'+\frac{\gamma^2+\gamma(m+\gamma-2)}{r^2}|g|^2\big)$, and integrating the middle term by parts against $r^{m-2b-1}$ gives $-\gamma(m+2\gamma-2b-2)\int|g|^2r^{m+2\gamma-2b-3}$, so that the coefficient of $\int|g|^2r^{m+2\gamma-2b-3}$ is $\gamma^2+\gamma(m+\gamma-2)-\gamma(m+2\gamma-2b-2)=2\gamma b$. We have obtained
\begin{align}
&\int_0^\infty\Big|f''+\frac{m-1}rf'-\gamma(m+\gamma-2)\frac{f}{r^2}\Big|^2r^{m-2a-1}dr\notag\\
&\qquad=\int_0^\infty|g''|^2r^{m+2\gamma-2a-1}dr+(1+2a)(m+2\gamma-1)\int_0^\infty|g'|^2r^{m+2\gamma-2a-3}dr,\label{iden1}\\
&\int_0^\infty\Big(|f'|^2+\gamma(m+\gamma-2)\frac{|f|^2}{r^2}\Big)r^{m-2b-1}dr\notag\\
&\qquad=\int_0^\infty|g'|^2r^{m+2\gamma-2b-1}dr+2\gamma b\int_0^\infty |g|^2r^{m+2\gamma-2b-3}dr .\label{iden2}
\end{align}
Next, for $1+a-b>0$ we make the change of variable and of function
\begin{equation}\label{s}
s=\frac{r^{1+a-b}}{1+a-b},\qquad h(s)=g(r),
\end{equation}
so that $g'(r)=h'(s)r^{a-b}$ and $g''(r)=\big(h''(s)+\frac{(a-b)\,h'(s)}{(1+a-b)s}\big)r^{2(a-b)}$.
Write $\kappa:=1+a-b$, so that $r^\kappa=\kappa s$, $dr=(\kappa s)^{\frac1\kappa-1}ds$ and, for every $p>-1$,
\begin{equation}\label{rpower}
\int_0^\infty\phi\,r^{p}dr=\kappa^{\frac{p+1}{\kappa}-1}\int_0^\infty\phi\,s^{\frac{p+1}{\kappa}-1}ds .
\end{equation}
Setting $\theta:=\frac{m+2\gamma+a-3b-1}{\kappa}$, one has $\frac{m+2\gamma-2a-4}{\kappa}=\theta-3$ and $\frac{m+2\gamma-2b-2}{\kappa}=\theta-1$, and \eqref{rpower} gives the four identities
\begin{align*}
\int_0^\infty|g''|^2r^{m+2\gamma-2a-1}dr&=(1+a-b)^{\theta}\int_0^\infty|h''|^2s^{\theta}ds\\
&\qquad-(1+a-b)^{\theta-2}(a-b)(m+2\gamma-a-b-2)\int_0^\infty|h'|^2s^{\theta-2}ds,\\
\int_0^\infty|g'|^2r^{m+2\gamma-2a-3}dr&=(1+a-b)^{\theta-2}\int_0^\infty|h'|^2s^{\theta-2}ds,\\
\int_0^\infty|g'|^2r^{m+2\gamma-2b-1}dr&=(1+a-b)^{\theta}\int_0^\infty|h'|^2s^{\theta}ds,\\
\int_0^\infty |g|^2r^{m+2\gamma-2b-3}dr&=(1+a-b)^{\theta-2}\int_0^\infty |h|^2s^{\theta-2}ds .
\end{align*}
Indeed, the last three follow directly from \eqref{rpower} and $g'(r)=h'(s)r^{\kappa-1}$; for the first one, $|g''|^2=|h''|^2r^{4\kappa-4}+2(\kappa-1)h''h'r^{3\kappa-4}+(\kappa-1)^2|h'|^2r^{2\kappa-4}$, the three terms contribute $\kappa^\theta\int|h''|^2s^\theta$, $(\kappa-1)\kappa^{\theta-1}\int(|h'|^2)'s^{\theta-1}ds=-(\kappa-1)(\theta-1)\kappa^{\theta-1}\int|h'|^2s^{\theta-2}$ and $(\kappa-1)^2\kappa^{\theta-2}\int|h'|^2s^{\theta-2}$ respectively, and $(\kappa-1)\big[(\kappa-1)-\kappa(\theta-1)\big]=-(a-b)(m+2\gamma-a-b-2)$. Combining these with \eqref{iden1} and \eqref{iden2}, and using $(1+2a)(m+2\gamma-1)-(a-b)(m+2\gamma-a-b-2)=(1+a+b)(m+2\gamma+a-b-1)$, we arrive at the key identity
\begin{align}\label{keyiden}
&\int_0^\infty\Big|f''+\frac{m-1}rf'-\gamma(m+\gamma-2)\frac{f}{r^2}\Big|^2r^{m-2a-1}dr\notag\\
&\qquad+\int_0^\infty\Big(|f'|^2+\gamma(m+\gamma-2)\frac{|f|^2}{r^2}\Big)r^{m-2b-1}dr\notag\\
&\quad=(1+a-b)^{\theta}\Big[\int_0^\infty|h''|^2s^{\theta}ds+\int_0^\infty|h'|^2s^{\theta}ds\notag\\
&\qquad\qquad+\frac{(1+a+b)(m+2\gamma+a-b-1)}{(1+a-b)^2}\int_0^\infty|h'|^2s^{\theta-2}ds\notag\\
&\qquad\qquad+\frac{2\gamma b}{(1+a-b)^2}\int_0^\infty |h|^2s^{\theta-2}ds\Big].
\end{align}
It is worth noting that the third integral on the right-hand side carries the factor $1+a+b$: on the line $1+a+b=0$ it disappears, and \eqref{keyiden} becomes exactly the functional $Q+P_1$ of \eqref{QP} with $\mu=\theta-1$. This is the reason why the line $1+a+b=0$ is the natural setting of Theorem \ref{thm:2nd}.

\section{Solenoidal fields: the Fourier reduction of the poloidal blocks}\label{sec:sol}

\subsection{The poloidal--toroidal decomposition}
We recall the framework of \cite{Ham23}, to which we refer for the details and for the precise function spaces. Let $\sigma=x/r$, $\partial_r=\sigma\cdot\nabla$, $\partial_r'=\partial_r+\frac{N-1}{r}$, and let $\nabla_\sigma$ and $\Delta_\sigma$ be the gradient and the Laplace--Beltrami operator on $\bS^{N-1}$. A vector field $\bU$ is \emph{toroidal} if $\bU\cdot x=0$ and $\dv \bU=0$, and \emph{poloidal} if it is solenoidal and of the form
\[
\bU=\bD f:=\big(\sigma\Delta_\sigma-r\partial_r'\nabla_\sigma\big)f
\]
for a scalar function $f$ with $\int_{\bS^{N-1}}f(r\sigma)d\sigma=0$ for all $r>0$, the \emph{poloidal potential}. Let us recall the main features of this decomposition, all proved in \cite[Section 2]{Ham23} and the references therein. Every vector field on $\R^N\setminus\{0\}$ splits as $\bU=\sigma U_R+\bU_S$ with $U_R=\sigma\cdot\bU$ its radial component and $\bU_S=\bU-\sigma U_R$ its spherical part, tangent to the spheres. A field is called \emph{pre-poloidal} if its spherical part is a spherical gradient, $\bU_S=\nabla_\sigma g$ for some scalar $g$, and \emph{poloidal} if it is pre-poloidal and solenoidal; it is \emph{toroidal} if $U_R=0$ and $\dv\bU=0$. The poloidal fields are exactly the fields $\bD f$ above, and the potential $f=\Delta_\sigma^{-1}U_R$ is determined by $\bU$. The basic facts are: (a) every smooth solenoidal field on $\R^N$ decomposes uniquely as $\bU=\bU_P+\bU_T$ with $\bU_P$ poloidal and $\bU_T$ toroidal; (b) for a poloidal $\bV$ and a toroidal $\bW$, $\int_{\bS^{N-1}}\bV\cdot\bW\,d\sigma=\int_{\bS^{N-1}}\nabla\bV\cdot\nabla\bW\,d\sigma=0$ on every sphere; (c) the spaces of poloidal and toroidal fields are stable under multiplication by radial functions, under $\partial_r$ and under $\Delta_\sigma$, so that they are also stable under the expansion in spherical harmonics, and the poloidal fields whose potential lies in different spherical modes are again orthogonal, with their gradients, on every sphere.  Expanding the poloidal potential in spherical harmonics, $f=\sum_{k\ge1}\sum_lf_{kl}(r)\phi_{kl}(\sigma)$, one has $\bU_P=\sum_{k,l}\bU_{kl}$ with $\bU_{kl}=\bD(f_{kl}\phi_{kl})$, and again all the fields $\bU_{kl}$, $\bU_T$ and their gradients are mutually orthogonal in $L^2(\bS^{N-1})$. We recall also the following elementary fact, which is used in Section \ref{sec:tor}: a toroidal field has vanishing spherical means, $\int_{\bS^{N-1}}\bU(r\sigma)d\sigma=0$ for all $r>0$. Indeed, $\partial_j(x_iU_j)=U_i+x_i\dv\bU=U_i$, so by the divergence theorem $\int_{|x|<r}U_i\,dx=\int_{|x|=r}x_i\,\bU\cdot\sigma\,d\sigma=0$ for all $r$, and differentiating in $r$ gives the claim. Consequently the three integrals of \eqref{HUPsol}, and therefore the deficit \eqref{deficit}, split as sums over the blocks:
\begin{equation}\label{split}
\delta_{2,N}(\bU)=\delta_{2,N}(\bU_T)+\sum_{k\ge1}\sum_l\delta_{2,N}(\bU_{kl}).
\end{equation}

\subsection{The one-dimensional representation of a block}
Fix $k\ge1$ and $l$, and write $f_{kl}(r)=r^{k-1}g_{kl}(r)$. Hamamoto \cite[Section 4]{Ham23} computed the three integrals of a block in terms of $g_{kl}$: there is a constant $c_k>0$ depending only on $N$ and $k$ such that
\begin{align}
\int_{\R^N}|\bU_{kl}|^2|x|^{\beta}dx&=c_k|\bS^{N-1}|\Big[\int_0^\infty|g_{kl}'|^2r^{N+2k+\beta-1}dr\notag\\
&\qquad\qquad-(N+k-2)\beta\int_0^\infty |g_{kl}|^2r^{N+2k+\beta-3}dr\Big],\label{Ubeta}\\
\int_{\R^N}|\nabla \bU_{kl}|^2dx&=c_k|\bS^{N-1}|\Big[\int_0^\infty|g_{kl}''|^2r^{N+2k-1}dr+(N+2k-1)\int_0^\infty|g_{kl}'|^2r^{N+2k-3}dr\Big].\label{gradU}
\end{align}
(In the variable $x=r^2$ these are the functionals $4Q$, $P_1$ and $P_0$ of \eqref{QP} with $\mu=\frac{N+2k}{2}$ and $\epsilon=\frac{N+k-2}{2}$, which explains the constant $C(N,k)=1+\sqrt{\mu^2-4\epsilon}$ in \eqref{CNk}.) For the reader's convenience we indicate how \eqref{Ubeta}--\eqref{gradU} can be checked directly in the case $k=1$, which is the one governing the sharp constant. A poloidal field of degree one with axis $\bc\in\R^N$, $|\bc|=1$, has the form
\[
\bU=\big(P(r)-Q(r)\big)(\bc\cdot\sigma)\,\sigma+Q(r)\,\bc,\qquad\sigma=\frac xr ,
\]
with $\dv\bU=\big[P'+\frac{N-1}{r}(P-Q)\big](\bc\cdot\sigma)$, so that $\bU$ is solenoidal if and only if $Q=\frac{rP'+(N-1)P}{N-1}$. Writing $u=\bc\cdot\sigma$ and using $\partial_j\sigma_i=\frac{\delta_{ij}-\sigma_i\sigma_j}{r}$, $\partial_ju=\frac{c_j-u\sigma_j}{r}$, one finds
\[
\partial_jU_i=(P'-Q')u\sigma_i\sigma_j+\frac{P-Q}{r}\sigma_i(c_j-u\sigma_j)+\frac{P-Q}{r}u(\delta_{ij}-\sigma_i\sigma_j)+Q'c_i\sigma_j ,
\]
and, computing the Frobenius norm with $\sigma\cdot(\bc-u\sigma)=0$ and $|\bc-u\sigma|^2=1-u^2$, all the mixed terms vanish except the one between the first and the last pieces. Integrating over the sphere of radius $r$, with $\int_{\bS^{N-1}}u^2d\sigma=\frac{|\bS^{N-1}|}{N}$ and $\int_{\bS^{N-1}}(1-u^2)d\sigma=\frac{N-1}{N}|\bS^{N-1}|$, this gives
\begin{align*}
\int_{\bS^{N-1}}|\bU(r\sigma)|^2d\sigma&=\frac{|\bS^{N-1}|}{N}\Big[|P(r)|^2+(N-1)|Q(r)|^2\Big],\\
\int_{\bS^{N-1}}|\nabla\bU(r\sigma)|^2d\sigma&=\frac{|\bS^{N-1}|}{N}\Big[|P'(r)|^2+(N-1)|Q'(r)|^2+\frac{2(N-1)}{r^2}|P(r)-Q(r)|^2\Big] .
\end{align*}
Substituting $P=-(N-1)g_{1l}$, hence $Q=-(N-1)g_{1l}-rg_{1l}'$ and $P-Q=rg_{1l}'$, integrating against $r^{N-1+\beta}dr$ and integrating the terms $g_{1l}g_{1l}'$ and $g_{1l}'g_{1l}''$ by parts, one recovers \eqref{Ubeta} and \eqref{gradU} for $k=1$, with $c_1|\bS^{N-1}|=\frac{N-1}{N}|\bS^{N-1}|$: for instance, in $\int\big[|P|^2+(N-1)|Q|^2\big]r^{N-1}dr$ the terms $N(N-1)^2|g_{1l}|^2+2(N-1)^2rg_{1l}g_{1l}'$ integrate to zero, leaving $(N-1)\int|g_{1l}'|^2r^{N+1}dr$, which is the case $\beta=0$ of \eqref{Ubeta}. We omit the details for $k\ge2$, which require the calculus of spherical harmonics, and refer to \cite{Ham23}.

The key observation of this paper is that \eqref{Ubeta}--\eqref{gradU} are exactly the integrals of the radial function
\begin{equation}\label{G}
G_{kl}(X):=g_{kl}(|X|),\qquad X\in\R^{N+2k},
\end{equation}
on the higher dimensional space $\R^{N+2k}$. Indeed, by \eqref{iden1} with $m=N+2k$, $\gamma=0$, $a=0$, and by \eqref{iden2},
\begin{equation}\label{lift}
\begin{aligned}
\int_{\R^N}|\nabla \bU_{kl}|^2dx&=\kappa_k\int_{\R^{N+2k}}|\Delta G_{kl}|^2dX,\qquad
\int_{\R^N}|\bU_{kl}|^2dx=\kappa_k\int_{\R^{N+2k}}|\nabla G_{kl}|^2dX,\\
\int_{\R^N}|\bU_{kl}|^2|x|^2dx&=\kappa_k\Big[\int_{\R^{N+2k}}|\nabla G_{kl}|^2|X|^2dX-2(N+k-2)\int_{\R^{N+2k}}|G_{kl}|^2dX\Big],
\end{aligned}
\end{equation}
with $\kappa_k:=c_k|\bS^{N-1}|/|\bS^{N+2k-1}|$. Hence
\begin{align}\label{blockG}
\delta_{2,N}(\bU_{kl})=\kappa_k\Big[&\int_{\R^{N+2k}}|\Delta G_{kl}|^2dX+\int_{\R^{N+2k}}|\nabla G_{kl}|^2|X|^2dX\notag\\
&-2(N+k-2)\int_{\R^{N+2k}}|G_{kl}|^2dX-2C_N\int_{\R^{N+2k}}|\nabla G_{kl}|^2dX\Big].
\end{align}

\subsection{The Fourier transform}
Let $H_{kl}=\widehat G_{kl}$ be the Fourier transform of $G_{kl}$ on $\R^{N+2k}$,
\[
H_{kl}(\xi)=\int_{\R^{N+2k}}e^{-iX\cdot\xi}G_{kl}(X)dX ,
\]
which is again radial (see \cite[Chapter IV]{SW71} for the Fourier transform of radial functions). By Plancherel's identity, the bracket in \eqref{blockG} equals
\begin{align}\label{Fourier}
(2\pi)^{-N-2k}\Big[\int|H_{kl}|^2|\xi|^4d\xi+\int|\nabla H_{kl}|^2|\xi|^2d\xi&-2(N+k-2)\int|H_{kl}|^2d\xi\notag\\
&-2C_N\int|H_{kl}|^2|\xi|^2d\xi\Big],
\end{align}
all integrals being over $\R^{N+2k}$. Now write $H_{kl}(\xi)=|\xi|^{\alpha}K_{kl}(\xi)$ with $\alpha\in\R$ to be chosen. Since $H_{kl}$ and $K_{kl}$ are radial, $\int|\nabla H_{kl}|^2|\xi|^2d\xi=|\bS^{N+2k-1}|\int_0^\infty|H_{kl}'|^2\rho^{N+2k+1}d\rho$, and \eqref{iden2} with $m=N+2k$, $\gamma=\alpha$, $b=-1$, applied to $H_{kl}=\rho^\alpha K_{kl}$, gives
\begin{align*}
\int_0^\infty\Big(|H_{kl}'|^2+\alpha(N+2k+\alpha-2)\frac{|H_{kl}|^2}{\rho^2}\Big)\rho^{N+2k+1}d\rho&=\int_0^\infty|K_{kl}'|^2\rho^{N+2k+2\alpha+1}d\rho\\
&\qquad-2\alpha\int_0^\infty|K_{kl}|^2\rho^{N+2k+2\alpha-1}d\rho ,
\end{align*}
i.e., since $|H_{kl}|^2\rho^{N+2k-1}=|K_{kl}|^2\rho^{N+2k+2\alpha-1}$,
\begin{equation}\label{HK}
\int|\nabla H_{kl}|^2|\xi|^2d\xi=\int|\nabla K_{kl}|^2|\xi|^{2+2\alpha}d\xi-\alpha(N+2k+\alpha)\int|K_{kl}|^2|\xi|^{2\alpha}d\xi .
\end{equation}
We choose $\alpha$ such that $\alpha(N+2k+\alpha)+2(N+k-2)=0$, namely
\begin{equation}\label{alpha}
\alpha=\alpha_k:=\frac{-(N+2k)+\sqrt{(N+2k)^2-8(N+k-2)}}{2}\in(-\tfrac{N+2k}{2},0),
\end{equation}
so that the zero order terms cancel, and \eqref{Fourier} becomes
\begin{equation}\label{firstorder}
(2\pi)^{-N-2k}\Big[\int_{\R^{N+2k}}|\nabla K_{kl}|^2|\xi|^{2+2\alpha}d\xi+\int_{\R^{N+2k}}|K_{kl}|^2|\xi|^{4+2\alpha}d\xi-2C_N\int_{\R^{N+2k}}|K_{kl}|^2|\xi|^{2+2\alpha}d\xi\Big].
\end{equation}
This is the deficit of a \emph{first order} weighted CKN inequality on $\R^{N+2k}$, for the radial function $K_{kl}$: with the notation of \eqref{CKNfirst} in dimension $N+2k$, the weights $|\xi|^{2+2\alpha}$ on the gradient and $|\xi|^{4+2\alpha}$ on the function correspond to the parameters $b=-1-\alpha$ and $a=-2-\alpha$, for which $1+a+b=-2-2\alpha$ and $1+b-a=2$, so that \eqref{firstorder} is the sum-form deficit of the inequality
\begin{equation}\label{CKN1}
\int_{\R^{N+2k}}|\nabla K|^2|\xi|^{2+2\alpha}d\xi+\int_{\R^{N+2k}}|K|^2|\xi|^{4+2\alpha}d\xi\ \ge\ 2C(N,k)\int_{\R^{N+2k}}|K|^2|\xi|^{2+2\alpha}d\xi .
\end{equation}
The sharp constant of \eqref{CKNfirst} is known \cite{CC09, CFL21}, and so is its sharp stability \cite{DLLN26}; rather than invoking these results for \eqref{CKN1}, we reduce \eqref{CKN1} directly to the weighted Gaussian Poincar\'e inequality of Theorem \ref{TA}, which gives at the same time the sharp constant and the sharp stability. More precisely, writing $K_{kl}(\xi)=e^{-|\xi|^2/2}w_{kl}(\xi)$, we have $\nabla K_{kl}=e^{-|\xi|^2/2}(\nabla w_{kl}-\xi w_{kl})$, hence
\[
|\nabla K_{kl}|^2|\xi|^{2+2\alpha}=\Big(|\nabla w_{kl}|^2-\xi\cdot\nabla(|w_{kl}|^2)+|\xi|^2|w_{kl}|^2\Big)e^{-|\xi|^2}|\xi|^{2+2\alpha},
\]
and, integrating by parts the middle term with $\dv\big(\xi\,e^{-|\xi|^2}|\xi|^{2+2\alpha}\big)=\big(N+2k+2+2\alpha-2|\xi|^2\big)e^{-|\xi|^2}|\xi|^{2+2\alpha}$, the terms containing $|\xi|^2|w_{kl}|^2$ cancel exactly against $\int|K_{kl}|^2|\xi|^{4+2\alpha}d\xi$; we obtain
\begin{align*}
\int|\nabla K_{kl}|^2|\xi|^{2+2\alpha}d\xi+\int|K_{kl}|^2|\xi|^{4+2\alpha}d\xi&=\int|\nabla w_{kl}|^2e^{-|\xi|^2}|\xi|^{2+2\alpha}d\xi\\
&\qquad+(N+2k+2+2\alpha)\int|w_{kl}|^2e^{-|\xi|^2}|\xi|^{2+2\alpha}d\xi ,
\end{align*}
and by \eqref{alpha} one has $N+2k+2+2\alpha_k=\sqrt{(N+2k)^2-8(N+k-2)}+2=2C(N,k)$. Therefore
\begin{align}\label{blockw}
\delta_{2,N}(\bU_{kl})=\frac{\kappa_k}{(2\pi)^{N+2k}}\Big[\int_{\R^{N+2k}}|\nabla w_{kl}|^2e^{-|\xi|^2}|\xi|^{2+2\alpha_k}d\xi&\notag\\
+2\big(C(N,k)-C_N\big)\int_{\R^{N+2k}}|w_{kl}|^2e^{-|\xi|^2}|\xi|^{2+2\alpha_k}d\xi\Big]&.
\end{align}
In particular, since $C(N,k)\ge C(N,1)=C_N$, this gives a one-line proof of Hamamoto's inequality $\delta_{2,N}(\bU_{kl})\ge0$ for each block. Let us make this explicit, as it recovers the poloidal part of Theorem \ref{THam}. By \eqref{blockw} with $C_N$ replaced by $C(N,k)$, the deficit of the block $k$ at the level $2C(N,k)$ is a nonnegative quadratic form, which vanishes exactly when $w_{kl}$ is constant, i.e.\ $H_{kl}=c\,e^{-|\xi|^2/2}|\xi|^{\alpha_k}$; hence $C(N,k)$ is the sharp constant of \eqref{HUPsol} on the poloidal fields of degree $k$, with the extremals computed in \eqref{gk} below. Moreover $C(N,k)=1+\frac12\sqrt{(N+2k)^2-8(N+k-2)}$ is increasing in $k$, since $(N+2k)^2-8(N+k-2)=4k^2+4k(N-2)+(N-4)^2$ is; hence $C_{P,N}=C(N,1)$, which is the value in \eqref{CPCT}. The toroidal constant $C_{T,N}=\frac{N+2}2$ is recovered in Section \ref{sec:tor} from \eqref{torV}, and $C_N=\min\{C_{P,N},C_{T,N}\}$ then follows from the orthogonality of the blocks exactly as in \cite{Ham23}. The first integral in \eqref{blockw} is the Dirichlet form of a weighted Gaussian Poincar\'e inequality, and Corollary \ref{cor:GP}(i), applied to the radial profile with $\alpha=N+2k+2+2\alpha_k=2C(N,k)>0$, yields
\begin{align}\label{Kstab}
\int_{\R^{N+2k}}|\nabla w_{kl}|^2e^{-|\xi|^2}|\xi|^{2+2\alpha_k}d\xi\ &\ge\ 4\inf_{c\in\R}\int_{\R^{N+2k}}|w_{kl}-c|^2e^{-|\xi|^2}|\xi|^{2+2\alpha_k}d\xi\notag\\
&=4\inf_{c\in\R}\int_{\R^{N+2k}}\big|H_{kl}-c\,e^{-\frac{|\xi|^2}{2}}|\xi|^{\alpha_k}\big|^2|\xi|^2d\xi ,
\end{align}
and the constant $4$ is sharp.

\subsection{Back to the poloidal extremals}
By Plancherel's identity and the second identity in \eqref{lift},
\begin{align*}
&\inf_{c\in\R}\int_{\R^{N+2k}}\big|H_{kl}-c\,e^{-\frac{|\xi|^2}{2}}|\xi|^{\alpha_k}\big|^2|\xi|^2d\xi\\
&\qquad=(2\pi)^{N+2k}\inf_{c\in\R}\int_{\R^{N+2k}}\big|\nabla\big(G_{kl}-c\,\cF^{-1}(e^{-|\cdot|^2/2}|\cdot|^{\alpha_k})\big)\big|^2dX .
\end{align*}
The inverse Fourier transform of $e^{-|\xi|^2/2}|\xi|^{\alpha}$ on $\R^{n}$ ($n=N+2k$) is a classical computation in terms of Kummer's function: for a radial function $\Phi(|\xi|)$ one has
\[
\cF^{-1}\Phi(X)=(2\pi)^{-n/2}|X|^{1-n/2}\int_0^\infty\Phi(\rho)J_{n/2-1}(|X|\rho)\rho^{n/2}d\rho
\]
\cite[Chapter IV]{SW71}, and the integral $\int_0^\infty\rho^{\alpha+n/2}e^{-\rho^2/2}J_{n/2-1}(|X|\rho)d\rho$ is evaluated by the formula of Step 4 in the proof of Theorem \ref{thm:1d}, with $m=\alpha+n/2$, $\nu=n/2-1$ and $\alpha$ (the parameter of the Gaussian there) equal to $\frac12$. This gives
\begin{align*}
\cF^{-1}\big(e^{-|\cdot|^2/2}|\cdot|^{\alpha}\big)(X)&=\frac{2^{\alpha/2}}{(2\pi)^{n/2}}\frac{\Gamma(\frac{n+\alpha}2)}{\Gamma(\frac n2)}\,{}_1F_1\Big(\frac{n+\alpha}{2},\frac n2,-\frac{|X|^2}{2}\Big)\\
&=\frac{2^{\alpha/2}}{(2\pi)^{n/2}}\frac{\Gamma(\frac{n+\alpha}2)}{\Gamma(\frac n2)}\,e^{-\frac{|X|^2}{2}}{}_1F_1\Big(-\frac{\alpha}{2},\frac n2,\frac{|X|^2}{2}\Big),
\end{align*}
where we used Kummer's transformation ${}_1F_1(p,q,z)=e^z{}_1F_1(q-p,q,-z)$. With $\alpha=\alpha_k$ and $n=N+2k$ we have $\frac{n+\alpha_k}{2}=\frac{N+2k+\sqrt{(N+2k)^2-8(N+k-2)}}{4}$, $-\frac{\alpha_k}{2}=\frac{N+2k-\sqrt{(N+2k)^2-8(N+k-2)}}{4}$ and $\frac n2=\frac{N+2k}{2}$, so that, explicitly,
\begin{align}
&\cF^{-1}\big(e^{-|\cdot|^2/2}|\cdot|^{\alpha_k}\big)(X)\notag\\
&\qquad=\gamma_{N,k}\,e^{-\frac{|X|^2}{2}}{}_1F_1\Big(\frac{N+2k-\sqrt{(N+2k)^2-8(N+k-2)}}{4},\frac{N+2k}{2},\frac{|X|^2}{2}\Big)\notag\\
&\qquad=\gamma_{N,k}\,g_k(|X|),\label{gk}\\
\gamma_{N,k}&:=\frac{2^{\frac{\alpha_k}{2}}}{(2\pi)^{\frac{N+2k}{2}}}\,\frac{\Gamma\Big(\frac{N+2k+\sqrt{(N+2k)^2-8(N+k-2)}}{4}\Big)}{\Gamma\big(\frac{N+2k}{2}\big)},\notag
\end{align}
where, by \eqref{fk},
\[
g_k(r):=r^{1-k}f_k(r)=e^{-\frac{r^2}{2}}{}_1F_1\Big(-\frac{\alpha_k}{2},\frac{N+2k}{2},\frac{r^2}{2}\Big).
\]
For $k=1$, one has $\alpha_1=\frac{-(N+2)+\sqrt{N^2-4N+12}}{2}$, $g_1=f_1$, and
\[
\gamma_{N,1}=2^{\frac{-(N+2)+\sqrt{N^2-4N+12}}{4}}(2\pi)^{-\frac{N+2}{2}}\,\frac{\Gamma\big(\frac{N+2+\sqrt{N^2-4N+12}}{4}\big)}{\Gamma\big(\frac{N+2}{2}\big)} .
\]
It is worth noting that $g_k$ is exactly Hamamoto's extremal \cite[Theorem 5.2]{Ham23} for the parameters $(\mu,\epsilon)=(\frac{N+2k}{2},\frac{N+k-2}{2})$ in the variable $x=r^2$, dilated as in \eqref{scaling} with $\lambda=2$, i.e.\ $g_k(r)=f_0(r^2/2)$: the Fourier transform reproduces Hamamoto's extremals without any computation, and the scale is the one dictated by the deficit \eqref{deficit}. Since the map $G_{kl}\mapsto\bU_{kl}$ is linear and, by the second identity in \eqref{lift}, isometric up to the factor $\kappa_k$, the distance transfers back to the block:
\begin{align*}
\inf_{c\in\R}\int_{\R^{N+2k}}\big|\nabla(G_{kl}-c\,g_k(|\cdot|))\big|^2dX
&=|\bS^{N+2k-1}|\inf_{c\in\R}\int_0^\infty|g_{kl}'-c\,g_k'|^2r^{N+2k-1}dr\\
&=\frac{1}{\kappa_k}\inf_{c\in\R}\int_{\R^N}\big|\bU_{kl}-c\,\bD(f_k\phi_{kl})\big|^2dx .
\end{align*}
Collecting \eqref{blockw}, \eqref{Kstab} and the last identities, we have proved the following.

\begin{proposition}[Block estimate]\label{prop:block}
Let $N\ge3$, $k\ge1$ and let $\bU_{kl}=\bD(f_{kl}\phi_{kl})$ be a poloidal field of degree $k$. Then
\begin{equation}\label{block}
\delta_{2,N}(\bU_{kl})\ \ge\ 4\inf_{c\in\R}\int_{\R^N}\big|\bU_{kl}-c\,\bD(f_k\phi_{kl})\big|^2dx+2\big(C(N,k)-C_N\big)\int_{\R^N}|\bU_{kl}|^2dx ,
\end{equation}
where $f_k$ and $C(N,k)$ are given in \eqref{fk} and \eqref{CNk}; in particular $\delta_{2,N}(\bU_{kl})\ge 2(C(N,k)-C_N)\int|\bU_{kl}|^2$, and the constant $4$ is sharp. The estimate holds for every $k\ge1$ with the same proof; for $k=1$ the second term vanishes, since $C(N,1)=C_N$, and \eqref{block} is the stability estimate inside the optimal block, while for $k\ge2$ both terms are used in the chains of Section \ref{sec:proofHUP}. Moreover, with $f_k^{(1)}$ the profile corresponding to the second level of the block (i.e.\ to $w_{kl}=|\xi|^2-c_0$ in the notation above),
\begin{align}\label{block2}
\delta_{2,N}(\bU_{kl})\ \ge\ &4\inf_{c\in\R}\int_{\R^N}\big|\bU_{kl}-c\,\bD(f_k\phi_{kl})\big|^2dx\notag\\
&+4\inf_{c,c'\in\R}\int_{\R^N}\big|\bU_{kl}-c\,\bD(f_k\phi_{kl})-c'\bD(f_k^{(1)}\phi_{kl})\big|^2dx\notag\\
&+2\big(C(N,k)-C_N\big)\int_{\R^N}|\bU_{kl}|^2dx .
\end{align}
\end{proposition}

\begin{proof}
Everything has been proved above, except \eqref{block2}, which follows in the same way from the two-term form of Corollary \ref{cor:GP}(i), the second level of the Laguerre expansion in the $w_{kl}$-variable being $w_{kl}=|\xi|^2-c_0$, whose inverse transform is, by \eqref{gk} and its analogue with the extra factor $|\xi|^2$, the profile $f_k^{(1)}$. Explicitly, since $g_k(r)=f_0(r^2/2)$ with $f_0$ the extremal \eqref{bf0} for the parameters $(\mu,\epsilon)=(\frac{N+2k}{2},\frac{N+k-2}{2})$, the second level is $f_k^{(1)}(r)=r^{k-1}f_1(r^2/2)$ with $f_1$ the function \eqref{f1} for the same parameters; this follows from Step 4 in the proof of Theorem \ref{thm:1d} and the dilation \eqref{scaling} with $\lambda=2$.
\end{proof}

Summing \eqref{block} over $l$ and using that $\{\bD(f_k\phi):\phi\in\cH_k\}$ is spanned by the $\bD(f_k\phi_{kl})$, we obtain for the full degree-$k$ part $\bU_k:=\sum_l\bU_{kl}$ of $\bU_P$:
\begin{equation}\label{blockk}
\delta_{2,N}(\bU_k)\ \ge\ 4\inf_{\phi\in\cH_k}\int_{\R^N}\big|\bU_k-\bD(f_k\phi)\big|^2dx+2\big(C(N,k)-C_N\big)\int_{\R^N}|\bU_k|^2dx .
\end{equation}
It is worth noting that the argument of this section is Theorem \ref{thm:1d} in the special case $2\mu=N+2k$, $4\epsilon=2(N+k-2)$, read on radial functions of $\R^{N+2k}$: the Fourier transform of radial functions is the Hankel transform of order $\frac{N+2k}2-1$, the passage from $H_{kl}$ to $K_{kl}$ is the substitution $w=w_0h$ of Step 3 in the proof of Theorem \ref{thm:1d}, and the inverse Fourier transform \eqref{gk} is the formula of Step 4. Conversely, Proposition \ref{prop:block} can be deduced from Theorem \ref{thm:1d} by the dilation $\lambda=2$ of \eqref{scaling}, under which the gap $2$ becomes the gap $4$ of \eqref{block}. We have kept the $N+2k$-dimensional formulation since it is the one which produces the extremal fields $\bD(f_k\phi)$ directly.

For later reference, and for comparison with the toroidal estimate of Section \ref{sec:tor}, let us record what \eqref{blockk} gives for the whole poloidal part $\bU_P=\sum_{k\ge1}\bU_k$ when only the first level of each block is used. Writing $\phi_{1l}=\sigma_l$, $l=1,\dots,N$, for the orthonormal basis of $\cH_1$ (up to the normalization $|\bS^{N-1}|^{-1/2}\sqrt N$), one has $\sum_lc_l\bD(f_1\phi_{1l})=\bD(\bc\cdot\sigma\,f_1)$ with $\bc=(c_1,\dots,c_N)$, and therefore
\begin{align}
\delta_{2,N}(\bU_P)&=\sum_{l=1}^N\delta_{2,N}(\bU_{1l})+\sum_{k\ge2}\sum_{l=1}^{\dim\cH_k}\delta_{2,N}(\bU_{kl})\notag\\
&\ge4\sum_{l=1}^N\inf_{c_l\in\R}\int_{\R^N}|\bU_{1l}-c_l\bD(f_1\phi_{1l})|^2dx+\sum_{k\ge2}\sum_{l=1}^{\dim\cH_k}2\big(C(N,k)-C_N\big)\int_{\R^N}|\bU_{kl}|^2dx\notag\\
&\ge\min\Big\{4,\ 2\min_{k\ge2}\big(C(N,k)-C_N\big)\Big\}\inf_{\bc\in\R^N}\int_{\R^N}\big|\bU_P-\bD(\bc\cdot\sigma\,f_1)\big|^2dx\notag\\
&=\Big(\sqrt{N^2+16}-\sqrt{N^2-4N+12}\Big)\inf_{\bc\in\R^N}\int_{\R^N}\big|\bU_P-\bD(\bc\cdot\sigma\,f_1)\big|^2dx ,\label{polstab}
\end{align}
since $2(C(N,2)-C_N)=\sqrt{N^2+16}-\sqrt{N^2-4N+12}<4$ and $C(N,k)$ is increasing in $k$. This is the sharp stability estimate of \eqref{HUPsol} restricted to poloidal fields, the constant being attained in the block $k=2$; it is stronger than what is needed in Theorem \ref{thm:HUP4}, whose first constant $\Lambda_1=N-\sqrt{N^2-4N+12}$ is smaller, since $N<\sqrt{N^2+16}$: the toroidal block, and not the poloidal one, dictates the stability constant of \eqref{HUPsol}.

We record the values we need: $2(C(N,1)-C_N)=0$, $2(C(N,2)-C_N)=\sqrt{N^2+16}-\sqrt{N^2-4N+12}$, and $2(C(N,k)-C_N)$ is increasing in $k$, with $2(C(N,3)-C_N)=\sqrt{N^2+4N+28}-\sqrt{N^2-4N+12}\ge4$ for $N\ge3$ (with equality at $N=3$).

\section{Toroidal fields: proof of Theorem \ref{thm:tor}}\label{sec:tor}

\subsection{Polynomial toroidal fields}
We first characterize the toroidal fields of the form $(Ax+\cB x\cdot x)e^{-|x|^2/2}$, i.e.\ the fields whose profile is a vector-valued polynomial of degree at most two (the Gaussian factor plays no role in the toroidality).

\begin{lemma}\label{lem:MN}
Let $A\in\R^{N\times N}$ and let $\cB=(B^1,\dots,B^N)$ be an $N$-tuple of symmetric matrices. Then the field $x\mapsto Ax$ is toroidal if and only if $A\in\mathcal A_N$, and the field $x\mapsto\cB x\cdot x-\frac12\Tr\cB$ is toroidal if and only if $\cB\in\cM_N$, i.e.\ \eqref{MN} holds; in that case $\Tr\cB=0$ and the field is $\cB x\cdot x$.
\end{lemma}

\begin{proof}
The first statement is clear: $\dv(Ax)=\Tr A$ and $Ax\cdot x=0$ for all $x$ if and only if $A$ is antisymmetric, and then $\Tr A=0$. For the second, put $\bV(x)=\cB x\cdot x-\frac12\Tr\cB$. One has $\dv \bV=2\sum_k\big(\sum_iB^i_{ik}\big)x_k$, so $\bV$ is divergence-free if and only if $\sum_iB^i_{ik}=0$ for all $k$. Next let $\phi(x):=x\cdot \bV(x)=\sum_{m,k,l}B^m_{kl}x_kx_lx_m-\frac12\sum_kx_k\Tr B^k$, a polynomial of degree three; its derivatives are
\[
\partial_i\phi=\sum_{k,l}\big(B^i_{kl}+B^k_{il}+B^l_{ki}\big)x_kx_l-\frac12\Tr B^i,\qquad
\partial^2_{ij}\phi=2\sum_l\big(B^i_{jl}+B^j_{il}+B^l_{ij}\big)x_l .
\]
If $\bV$ is toroidal, then $\phi\equiv0$, hence $\Tr B^i=0$ and $B^i_{jl}+B^j_{il}+B^l_{ij}=0$ for all $i,j,l$. Conversely, under \eqref{MN} the polynomial $\phi$ satisfies $\nabla^2\phi\equiv0$, $\nabla\phi(0)=0$ and $\phi(0)=0$, hence $\phi\equiv0$, so $\bV\cdot x=0$, and $\bV$ is divergence-free by the second condition.
\end{proof}

\subsection{The Gaussian Poincar\'e inequality}
Let $\bU$ be toroidal and write $\bU=\bV e^{-|x|^2/2}$. Then $\partial_jU_i=e^{-|x|^2/2}(\partial_jV_i-x_jV_i)$, so that $|\nabla\bU|^2=e^{-|x|^2}\big(|\nabla\bV|^2-x\cdot\nabla(|\bV|^2)+|x|^2|\bV|^2\big)$; integrating by parts the middle term, $-\int x\cdot\nabla(|\bV|^2)e^{-|x|^2}dx=\int|\bV|^2(N-2|x|^2)e^{-|x|^2}dx$, and the terms containing $|x|^2|\bV|^2$ cancel against $\int|\bU|^2|x|^2dx=\int|\bV|^2|x|^2e^{-|x|^2}dx$. We obtain
\begin{equation}\label{torV}
\int_{\R^N}|\nabla \bU|^2dx+\int_{\R^N}|\bU|^2|x|^2dx=\int_{\R^N}|\nabla \bV|^2e^{-|x|^2}dx+N\int_{\R^N}|\bV|^2e^{-|x|^2}dx ,
\end{equation}
so that the toroidal deficit of Theorem \ref{thm:tor} equals $\int|\nabla \bV|^2e^{-|x|^2}-2\int|\bV|^2e^{-|x|^2}$. Since $\bU$ is toroidal, $\int_{\bS^{N-1}}\bV(r\sigma)d\sigma=0$ for every $r>0$ (see Section \ref{sec:sol}), hence each component $V_i$ has zero mean with respect to the Gaussian measure $e^{-|x|^2}dx$. We use the classical Gaussian Poincar\'e inequality in the following form with two remainder terms. The Ornstein--Uhlenbeck operator $L:=-\Delta+2x\cdot\nabla$ is self-adjoint and nonnegative on $L^2(e^{-|x|^2}dx)$, with $\int|\nabla v|^2e^{-|x|^2}dx=\int vLv\,e^{-|x|^2}dx$; its eigenfunctions are the products $H_{m_1}(x_1)\cdots H_{m_N}(x_N)$ of Hermite polynomials, with eigenvalue $2(m_1+\dots+m_N)$, so that its eigenvalues are $0,2,4,6,\dots$, the eigenspace of $0$ consists of the constants, that of $2$ is spanned by $x_1,\dots,x_N$, and that of $4$ is spanned by $x_ix_j$ ($i<j$) and $x_i^2-\frac12$. Expanding $v$ in this orthogonal basis, the level $j$ contributes $2j$ times its squared norm to $\int|\nabla v|^2e^{-|x|^2}$, while the three terms on the right-hand side of \eqref{GP2} below are twice the squared norms of the projections of $v$ onto the levels $j\ge1$, $j\ge2$ and $j\ge3$ (recall that $v$ has zero mean); since $2j\ge2\min\{j,3\}$, one obtains: for every $v$ with $\int ve^{-|x|^2}dx=0$,
\begin{align}\label{GP2}
\int_{\R^N}|\nabla v|^2e^{-|x|^2}dx\ \ge\ &2\int_{\R^N}|v|^2e^{-|x|^2}dx+2\inf_{\bc\in\R^N}\int_{\R^N}|v-\bc\cdot x|^2e^{-|x|^2}dx\notag\\
&+2\inf_{\substack{\bc\in\R^N\\ B\ \text{symmetric}}}\int_{\R^N}\Big|v-\bc\cdot x-\Big(Bx\cdot x-\frac12\Tr B\Big)\Big|^2e^{-|x|^2}dx ,
\end{align}
with equality exactly when $v$ has no component in the levels $j\ge4$. It is worth noting that a fourth term, with the cubic polynomials of the level $3$, could be added with the same constant $2$, since $2j\ge8$ for $j\ge4$. Applying \eqref{GP2} to each component of $\bV$ and summing, we get
\begin{align}\label{Vstab}
\int_{\R^N}|\nabla \bU|^2dx&+\int_{\R^N}|\bU|^2|x|^2dx-(N+2)\int_{\R^N}|\bU|^2dx\notag\\
&\ge2\inf_{A\in\R^{N\times N}}\int_{\R^N}|\bV-Ax|^2e^{-|x|^2}dx\notag\\
&\qquad+2\inf_{\substack{A\in\R^{N\times N}\\ \cB\ \text{symmetric}}}\int_{\R^N}\Big|\bV-Ax-\Big(\cB x\cdot x-\frac12\Tr\cB\Big)\Big|^2e^{-|x|^2}dx .
\end{align}

\subsection{Reduction of the infima to $\mathcal A_N$ and $\cM_N$}
It remains to show that the infima in \eqref{Vstab} can be restricted to $A\in\mathcal A_N$ and $\cB\in\cM_N$; this is where the toroidality of $\bU$ is used a second time. We first observe that, since $\bU\cdot x=0$ and $\dv \bU=0$, for every polynomial $\varphi$ (more generally, for every smooth $\varphi$ of polynomial growth)
\begin{equation}\label{torid}
\int_{\R^N}\bU\cdot\nabla\varphi\,e^{-\frac{|x|^2}{2}}dx=\int_{\R^N}\bU\cdot\nabla\Big(\varphi e^{-\frac{|x|^2}{2}}\Big)dx+\int_{\R^N}(\bU\cdot x)\varphi\,e^{-\frac{|x|^2}{2}}dx=0 :
\end{equation}
indeed $\nabla(\varphi e^{-|x|^2/2})=(\nabla\varphi-x\varphi)e^{-|x|^2/2}$, the second integral vanishes because $\bU\cdot x=0$, and the first one vanishes after an integration by parts because $\dv\bU=0$.
Expanding the square and using $\int x_kx_le^{-|x|^2}dx=\frac{\gamma_N}{2}\delta_{kl}$ with $\gamma_N:=\int e^{-|x|^2}dx=\pi^{N/2}$, one has $\int|\bV-Ax|^2e^{-|x|^2}dx=\int|\bV|^2e^{-|x|^2}dx-2\sum_{k,l}A_{kl}\int V_kx_le^{-|x|^2}dx+\frac{\gamma_N}{2}\sum_{k,l}A_{kl}^2$, so the first infimum in \eqref{Vstab} is attained at the matrix $A_0$ with entries $(A_0)_{kl}=\frac{2}{\gamma_N}\int_{\R^N}V_kx_le^{-|x|^2}dx$, and its value is
\[
\inf_{A\in\R^{N\times N}}\int_{\R^N}|\bV-Ax|^2e^{-|x|^2}dx=\int_{\R^N}|\bV|^2e^{-|x|^2}dx-\frac{2}{\gamma_N}\sum_{k,l=1}^N\Big(\int_{\R^N}V_kx_le^{-|x|^2}dx\Big)^2 .
\] Taking $\varphi=x_i^2$ and $\varphi=x_kx_l$ ($k\ne l$) in \eqref{torid} gives $\int V_ix_ie^{-|x|^2}=0$ and $\int V_kx_le^{-|x|^2}=-\int V_lx_ke^{-|x|^2}$, i.e.\ $A_0\in\mathcal A_N$. This proves the first term of \eqref{stabtor}.

For the second infimum, by the orthogonality of the Hermite eigenspaces the two minimizations decouple, and the minimization over $\cB$ amounts to minimizing
\begin{align*}
-2\int_{\R^N}\bV\cdot(\cB x\cdot x)e^{-|x|^2}dx&+\int_{\R^N}\Big|\cB x\cdot x-\frac12\Tr\cB\Big|^2e^{-|x|^2}dx\\
&=\sum_{i=1}^N\Big(-2\Tr(B^iC^i)+\frac{\gamma_N}{2}\Tr\big((B^i)^2\big)\Big),
\end{align*}
where $C^i$ is the symmetric matrix with entries $C^i_{kl}=\int_{\R^N}V_ix_kx_le^{-|x|^2}dx$. Here we used $\int V_ie^{-|x|^2}dx=0$ for the term $\Tr B^i\int V_ie^{-|x|^2}$, and the following moments: after an orthogonal change of variables diagonalizing $B^i=\mathrm{diag}(\beta_1,\dots,\beta_N)$, which leaves $e^{-|x|^2}$ invariant, $\int(B^ix\cdot x)^2e^{-|x|^2}dx=\sum_k\beta_k^2\int y_k^4e^{-|y|^2}dy+\sum_{k\ne l}\beta_k\beta_l\int y_k^2y_l^2e^{-|y|^2}dy=\frac34\gamma_N\Tr((B^i)^2)+\frac14\gamma_N\big((\Tr B^i)^2-\Tr((B^i)^2)\big)$, and $\int(B^ix\cdot x)e^{-|x|^2}dx=\frac{\gamma_N}2\Tr B^i$, so that $\int|B^ix\cdot x-\frac12\Tr B^i|^2e^{-|x|^2}dx=\frac{\gamma_N}{2}\Tr((B^i)^2)$. Since $-2\Tr(B^iC^i)+\frac{\gamma_N}{2}\Tr((B^i)^2)=\frac{\gamma_N}{2}\Tr\big((B^i-\frac{2}{\gamma_N}C^i)^2\big)-\frac{2}{\gamma_N}\Tr((C^i)^2)$, the minimum over symmetric $B^i$ is attained exactly at $B^i=\frac{2}{\gamma_N}C^i$, and its value is $-\frac{2}{\gamma_N}\sum_i\Tr((C^i)^2)$. Finally we check that $\cC=(C^1,\dots,C^N)\in\cM_N$: the matrices $C^i$ are symmetric; $\sum_iC^i_{ik}=\int(\bU\cdot x)x_ke^{-|x|^2/2}=0$; taking $\varphi=|x|^2x_k$ in \eqref{torid} and using $\bU\cdot x=0$ gives $\Tr C^k=\int \bU_k|x|^2e^{-|x|^2/2}=0$; and taking $\varphi=x_ix_jx_l$ in \eqref{torid} gives $C^i_{jl}+C^j_{il}+C^l_{ij}=0$. Hence the second infimum in \eqref{Vstab} is attained at $(A_0,\frac{2}{\gamma_N}\cC)\in\mathcal A_N\times\cM_N$, and by Lemma \ref{lem:MN} we may drop the term $\frac12\Tr\cB$. This proves \eqref{stabtor}.

The sharpness of the constant $2$ in the first term follows by taking $\bV=\cB x\cdot x$ with $\cB\in\cM_N\setminus\{0\}$ (such $\cB$ exist for $N\ge3$: for $N=3$, one may take $B^1=e_2\otimes e_3+e_3\otimes e_2$, $B^2=-(e_1\otimes e_3+e_3\otimes e_1)$, $B^3=0$, which gives $\bV=2(x_2x_3,-x_1x_3,0)$, a toroidal field), for which \eqref{stabtor} is an equality with the first term nonzero and the second one zero. This completes the proof of Theorem \ref{thm:tor}. \qed

Combining \eqref{torV} with the Gaussian Poincar\'e inequality without remainder terms, one recovers Hamamoto's toroidal constant $C_{T,N}=\frac{N+2}{2}$ in two lines; the profile $e^{-|x|^2/2}$ is again the one dictated by the deficit.

\section{Stability of the solenoidal HUP: proofs of Theorems \ref{thm:prodHUP4}--\ref{thm:HUP3}}\label{sec:proofHUP}

\subsection{Proofs of Theorems \ref{thm:HUP4} and \ref{thm:HUP3}}
Let $\bU=\bU_T+\sum_{k\ge1}\bU_k$ be the decomposition of Section \ref{sec:sol}, so that $\delta_{2,N}(\bU)=\delta_{2,N}(\bU_T)+\sum_k\delta_{2,N}(\bU_k)$ by \eqref{split}. We write, for the families appearing in the theorems,
\begin{gather*}
d_T^2:=\inf_{A\in\mathcal A_N}\int\big|\bU_T-Axe^{-\frac{|x|^2}2}\big|^2,\qquad
d_{T}'^2:=\inf_{\substack{A\in\mathcal A_N\\ \cB\in\cM_N}}\int\big|\bU_T-(Ax+\cB x\cdot x)e^{-\frac{|x|^2}2}\big|^2,\\
d_k^2:=\inf_{\phi\in\cH_k}\int\big|\bU_k-\bD(f_k\phi)\big|^2 ,
\end{gather*}
and $\|\bU_k\|^2:=\int|\bU_k|^2$, $\|\bU_T\|^2:=\int|\bU_T|^2$. Since the blocks are mutually orthogonal in $L^2(\R^N)$ and the elements of $\cE_1,\cE_2,\cT_1,\cT_2$ belong to the corresponding blocks, the distances in \eqref{stabHUP4}--\eqref{stabHUP3} split; for instance $\dist(\bU,\cE_2+\cT_1)^2=d_T^2+d_1^2+d_2^2+\sum_{k\ge3}\|\bU_k\|^2$. We shall use the following two inputs:
\begin{itemize}
\item[(T)] by Theorem \ref{thm:tor} and $\delta_{2,N}(\bU_T)=\big[\int|\nabla \bU_T|^2+\int|\bU_T|^2|x|^2-(N+2)\int|\bU_T|^2\big]+(N+2-2C_N)\|\bU_T\|^2$,
\[
\delta_{2,N}(\bU_T)\ \ge\ (N+2-2C_N)\|\bU_T\|^2+2d_T^2+2d_T'^2,\qquad N+2-2C_N=N-\sqrt{N^2-4N+12};
\]
\item[(P$_k$)] by \eqref{blockk}, $\delta_{2,N}(\bU_k)\ge4d_k^2+2(C(N,k)-C_N)\|\bU_k\|^2$, and $d_k\le\|\bU_k\|$.
\end{itemize}

The proofs consist in a bookkeeping of the gaps: each block $B$ of $\bU$ comes with a lower bound of $\delta_{2,N}(B)$ by a combination of $\|B\|^2$ and of the distances of $B$ to its own extremals, given by (T) and (P$_k$), and the right-hand side of the theorem asks, from each block, a combination of the same quantities with the coefficients $\Lambda_j$; the inequalities between the two combinations are what we check. The reader may find it useful to keep in mind the following summary of the available gaps at $N\ge4$: the toroidal block provides $\Lambda_1$ on $\|\bU_T\|^2$ and $2+2$ on the two Hermite distances; the block $k=1$ provides $4$ on $d_1^2$; the block $k=2$ provides $\Lambda_1+\Lambda_2$ on $\|\bU_2\|^2$ and $4$ on $d_2^2$; the blocks $k\ge3$ provide at least $4$ on $\|\bU_k\|^2$. The four constants are then forced, in this order, by the toroidal block, by the block $k=2$, by the toroidal block again, and by the block $k=1$.

\begin{proof}[Proof of Theorem \ref{thm:HUP4}]
Let $N\ge4$; note that $\Lambda_1,\dots,\Lambda_4>0$ (since $N<\sqrt{N^2+16}<N+2$ and $N-2<\sqrt{N^2-4N+12}<N$), $\Lambda_1+\Lambda_2+\Lambda_3+\Lambda_4=4$ and $\Lambda_2+\Lambda_3=2$. The right-hand side of \eqref{stabHUP4} equals
\begin{align*}
\Lambda_1\Big(\|\bU_T\|^2+d_1^2+\sum_{k\ge2}\|\bU_k\|^2\Big)&+\Lambda_2\Big(d_T^2+d_1^2+\sum_{k\ge2}\|\bU_k\|^2\Big)\\
&+\Lambda_3\Big(d_T^2+d_1^2+d_2^2+\sum_{k\ge3}\|\bU_k\|^2\Big)\\
&+\Lambda_4\Big(d_T'^2+d_1^2+d_2^2+\sum_{k\ge3}\|\bU_k\|^2\Big),
\end{align*}
and we verify that it is bounded by $\delta_{2,N}(\bU)$ block by block, using (T) and (P$_k$).

\emph{Toroidal block.} The required amount is $\Lambda_1\|\bU_T\|^2+(\Lambda_2+\Lambda_3)d_T^2+\Lambda_4d_T'^2=(N+2-2C_N)\|\bU_T\|^2+2d_T^2+\Lambda_4d_T'^2\le\delta_{2,N}(\bU_T)$ by (T), since $\Lambda_4<2$.

\emph{Block $k=1$.} The required amount is $(\Lambda_1+\Lambda_2+\Lambda_3+\Lambda_4)d_1^2=4d_1^2\le\delta_{2,N}(\bU_1)$ by (P$_1$).

\emph{Block $k=2$.} The required amount is $(\Lambda_1+\Lambda_2)\|\bU_2\|^2+(\Lambda_3+\Lambda_4)d_2^2$. Now $\Lambda_1+\Lambda_2=\sqrt{N^2+16}-\sqrt{N^2-4N+12}=2(C(N,2)-C_N)$ and $\Lambda_3+\Lambda_4=4+\sqrt{N^2-4N+12}-\sqrt{N^2+16}<4$, so the required amount is at most $\delta_{2,N}(\bU_2)$ by (P$_2$).

\emph{Blocks $k\ge3$.} The required amount is $4\|\bU_k\|^2$, and $\delta_{2,N}(\bU_k)\ge2(C(N,k)-C_N)\|\bU_k\|^2\ge2(C(N,3)-C_N)\|\bU_k\|^2\ge4\|\bU_k\|^2$ by (P$_k$) and the monotonicity recorded after \eqref{blockk}.

Summing up the four cases gives \eqref{stabHUP4}.

For the sharpness we exhibit, for each $j$, a field for which equality holds with exactly the first $j$ terms nonzero; since these terms are then all equal to $\|\bU\|^2$, the right-hand side equals $(\Lambda_1+\dots+\Lambda_j)\|\bU\|^2$, and increasing $\Lambda_j$ alone would violate \eqref{stabHUP4}. For $j=1$ take $\bU=Axe^{-|x|^2/2}\in\cT_1$, $A\ne0$: the toroidal deficit at level $N+2$ vanishes, so $\delta_{2,N}(\bU)=(N+2-2C_N)\|\bU\|^2=\Lambda_1\|\bU\|^2$, while $\dist(\bU,\cE_1)^2=\|\bU\|^2$ and the three other distances vanish. For $j=2$ take $\bU=\bD(f_2\phi_2)\in\cE_2$, $\phi_2\ne0$: by Proposition \ref{prop:block} with equality (the block deficit at level $2C(N,2)$ vanishes), $\delta_{2,N}(\bU)=2(C(N,2)-C_N)\|\bU\|^2=(\Lambda_1+\Lambda_2)\|\bU\|^2$, the distances to $\cE_1$ and to $\cE_1+\cT_1$ equal $\|\bU\|^2$, and the distances to $\cE_2+\cT_1$ and $\cE_2+\cT_2$ vanish. For $j=3$ take $\bU=\cB x\cdot x\,e^{-|x|^2/2}\in\cT_2$ with $\cB\in\cM_N\setminus\{0\}$: by \eqref{torV} and the Hermite expansion, $\delta_{2,N}(\bU)=(N+2-2C_N)\|\bU\|^2+2\|\bU\|^2=(\Lambda_1+\Lambda_2+\Lambda_3)\|\bU\|^2$, the first three distances equal $\|\bU\|^2$ (a cubic profile is $L^2$-orthogonal to the linear profiles by parity), and the fourth vanishes. For $j=4$ take the poloidal field of degree one which corresponds, through the Fourier reduction of Section \ref{sec:sol}, to $w_{11}=|\xi|^2-c_0$ with $c_0$ chosen so that $w_{11}$ is orthogonal to constants in $L^2(e^{-|\xi|^2}|\xi|^{2+2\alpha_1}d\xi)$, i.e.\ to the second level of the block: by the equality case of Corollary \ref{cor:GP}(i) one has $\delta_{2,N}(\bU)=4\|\bU\|^2=(\Lambda_1+\Lambda_2+\Lambda_3+\Lambda_4)\|\bU\|^2$, while all four distances equal $\|\bU\|^2$ since $\bU$ is orthogonal to $\bD(f_1\phi_1)$ (the first level of the same block) and to the other blocks.
\end{proof}

\begin{proof}[Proof of Theorem \ref{thm:HUP3}]
Let $N=3$, so that $2C_3=5=N+2$, $2(C(3,2)-C_3)=2$ and $2(C(3,3)-C_3)=4$. The right-hand side of \eqref{stabHUP3} equals
\[
2\Big(d_T^2+d_1^2+\sum_{k\ge2}\|\bU_k\|^2\Big)+2\Big(d_T'^2+d_1^2+d_2^2+\sum_{k\ge3}\|\bU_k\|^2\Big).
\]
Toroidal block: $2d_T^2+2d_T'^2\le\delta_{2,3}(\bU_T)$ by (T), since now $N+2-2C_3=0$. Block $k=1$: $4d_1^2\le\delta_{2,3}(\bU_1)$ by (P$_1$). Block $k=2$: $2\|\bU_2\|^2+2d_2^2\le2(C(3,2)-C_3)\|\bU_2\|^2+4d_2^2\le\delta_{2,3}(\bU_2)$ by (P$_2$). Blocks $k\ge3$: $4\|\bU_k\|^2\le2(C(3,k)-C_3)\|\bU_k\|^2\le\delta_{2,3}(\bU_k)$. This proves \eqref{stabHUP3}. For the sharpness of the first constant, take $\bU=\bD(f_2\phi)$ with $\phi\in\cH_2\setminus\{0\}$: by Proposition \ref{prop:block} applied with $C(3,2)$ in place of $C_3$, the block deficit at level $2C(3,2)$ vanishes, so $\delta_{2,3}(\bU)=2(C(3,2)-C_3)\|\bU\|^2=2\|\bU\|^2$, while $\dist(\bU,\cE_1+\cT_1)^2=\|\bU\|^2$ and $\dist(\bU,\cE_2+\cT_2)^2=0$; the field $\bU=\cB x\cdot x\,e^{-|x|^2/2}$, $\cB\in\cM_3\setminus\{0\}$, does the same job. For the sharpness of the second constant, take the second-level field of the degree-one block as in the proof of Theorem \ref{thm:HUP4}: $\delta_{2,3}(\bU)=4\|\bU\|^2$ and both distances equal $\|\bU\|^2$.
\end{proof}

\subsection{Proofs of Theorems \ref{thm:prodHUP4} and \ref{thm:prodHUP3}}\label{sec:prod}
The passage from the fixed scale to the product form rests on the following elementary principle.

\begin{lemma}\label{lem:transfer}
Let $A,B,C$ be nonnegative functionals on a class of functions $\mathcal X$ stable under a one-parameter group of dilations $u\mapsto u_\lambda$, $\lambda>0$, such that $A(u_\lambda)=\lambda^{p}A(u)$, $B(u_\lambda)=\lambda^{-p}B(u)$ and $C(u_\lambda)=C(u)$ for some $p>0$, and let $\cE\subset\mathcal X$. If
\[
A(u)+B(u)-2C_*C(u)\ \ge\ \kappa\,\dist(u,\cE)^2\qquad\text{for all }u\in\mathcal X ,
\]
where $\dist$ is computed in a norm for which the dilations are isometries, then
\[
\sqrt{A(u)B(u)}-C_*C(u)\ \ge\ \frac\kappa2\,\dist\big(u,\cE_{1/\lambda_*}\big)^2\ \ge\ \frac\kappa2\,\dist\big(u,\cE^{\rm dil}\big)^2,\qquad\lambda_*:=\Big(\frac{B(u)}{A(u)}\Big)^{\frac1{2p}},
\]
where $\cE_\rho:=\{E_\rho:E\in\cE\}$ and $\cE^{\rm dil}:=\bigcup_{\rho>0}\cE_\rho$.
\end{lemma}

\begin{proof}
For $u_{\lambda_*}$ one has $A(u_{\lambda_*})=B(u_{\lambda_*})=\sqrt{A(u)B(u)}$, so the hypothesis applied to $u_{\lambda_*}$ reads $2\sqrt{A(u)B(u)}-2C_*C(u)\ge\kappa\dist(u_{\lambda_*},\cE)^2$, and $\dist(u_{\lambda_*},\cE)=\dist(u,\cE_{1/\lambda_*})$ since $u\mapsto u_\lambda$ is an isometry with inverse $u\mapsto u_{1/\lambda}$; the last inequality is $\cE_{1/\lambda_*}\subset\cE^{\rm dil}$.
\end{proof}

We apply the lemma with $p=2$ to the three integrals of \eqref{HUPsol} (and, in Section \ref{sec:2nd}, with $p=2+2a$ to those of \eqref{2ndCKN}). For a vector field $\bU$ and $\lambda>0$ let $\bU_\lambda(x):=\lambda^{N/2}\bU(\lambda x)$, which is solenoidal (resp.\ toroidal, poloidal of degree $k$) whenever $\bU$ is. Then $\int|\bU_\lambda|^2=\int|\bU|^2$, $\int|\nabla\bU_\lambda|^2=\lambda^2\int|\nabla\bU|^2$ and $\int|\bU_\lambda|^2|x|^2=\lambda^{-2}\int|\bU|^2|x|^2$. With $A,B,C$ the three integrals of $\bU$, the choice
\[
\lambda_*:=\Big(\frac{B}{A}\Big)^{\frac14}
\]
gives $A(\bU_{\lambda_*})=B(\bU_{\lambda_*})=\sqrt{AB}$, hence $\delta_{2,N}(\bU_{\lambda_*})=2\sqrt{AB}-2C_NC=2\delta_{1,N}(\bU)$, by the invariance of $\delta_{1,N}$. Moreover, for every family $\cE$ and every $\lambda>0$, $\dist(\bU_\lambda,\cE)=\dist(\bU,\cE_{1/\lambda})$, where $\cE_\rho:=\{\bE_\rho:\bE\in\cE\}$, since $\bU\mapsto\bU_\lambda$ is an isometry of $L^2(\R^N)$ with inverse $\bU\mapsto\bU_{1/\lambda}$. Applying Theorem \ref{thm:HUP4} to $\bU_{\lambda_*}$ and dividing by $2$ we obtain
\begin{align*}
\sqrt{AB}-C_NC\ \ge\ \tfrac12\Big[&\Lambda_1\dist(\bU,(\cE_1)_{1/\lambda_*})^2+\Lambda_2\dist(\bU,(\cE_1+\cT_1)_{1/\lambda_*})^2\\
&+\Lambda_3\dist(\bU,(\cE_2+\cT_1)_{1/\lambda_*})^2+\Lambda_4\dist(\bU,(\cE_2+\cT_2)_{1/\lambda_*})^2\Big],
\end{align*}
which is the stronger form of \eqref{prod4} mentioned in Theorem \ref{thm:prodHUP4}; since $\cE_{1/\lambda_*}\subset\cE^{\rm dil}$, each distance on the right-hand side is at least the distance to $\cE^{\rm dil}$, and \eqref{prod4} follows. Theorem \ref{thm:HUP3} gives \eqref{prod3} in the same way. This proves Theorems \ref{thm:prodHUP4} and \ref{thm:prodHUP3}, up to the sharpness.

For the sharpness in Theorems \ref{thm:prodHUP4} and \ref{thm:prodHUP3}, we observe that the test fields used for the first constants in Theorems \ref{thm:HUP4} and \ref{thm:HUP3} are balanced, i.e.\ satisfy $A=B$: they minimize the ratio of the deficit $\delta_{2,N}$ to the dilation invariant $L^2$-norm within their block, and $\frac{d}{d\lambda}\delta_{2,N}(\bU_\lambda)|_{\lambda=1}=2(A-B)$ vanishes at such a minimizer (the field $\bU=Mxe^{-|x|^2/2}$ with $M\in\cA_N$, for instance, has $A=B$ by a direct computation: with $M$ the rotation in the $(x_1,x_2)$-plane, $|Mx|^2=x_1^2+x_2^2$ has spherical average $\frac{2r^2}N$, and $|\nabla\bU|^2=e^{-|x|^2}\big(|M|^2-2|Mx|^2+|Mx|^2|x|^2\big)$ with $|M|^2=2$, so that both $A$ and $B$ equal $|\bS^{N-1}|\int_0^\infty\frac{2}{N}r^{N+3}e^{-r^2}dr$ after the term $\int(2-\frac{4r^2}{N})r^{N-1}e^{-r^2}dr=0$ is discarded). For a balanced field $\delta_{1,N}(\bU)=\frac12\delta_{2,N}(\bU)$, and the distance to $\cE^{\rm dil}$ equals the distance to $\cE$, as the test fields are orthogonal to the dilated families as well (the orthogonality of the blocks, and of the Hermite and Laguerre levels, is preserved by dilations of the family). Hence the equalities of the sum-form case become equalities of the product-form case with half the constants. \qed

The transition between $N\ge4$ and $N=3$ is visible in the constants: evaluating the formulas for $\Lambda_1,\dots,\Lambda_4$ at $N=3$ gives $\Lambda_1=\Lambda_3=0$ and $\Lambda_2=\Lambda_4=2$, i.e.\ exactly the chain \eqref{stabHUP3}. The two constants which vanish are those dictated by the toroidal block, and correspondingly the families $\cE_1$ and $\cE_1+\cT_1$ (resp.\ $\cE_2+\cT_1$ and $\cE_2+\cT_2$) play the same role at $N=3$, where the toroidal fields $Axe^{-|x|^2/2}$ become extremals. Finally, the chains of Theorems \ref{thm:HUP4} and \ref{thm:HUP3} can be continued in the same way, using the next Hermite level (toroidal fields with cubic profiles), the block $k=3$ and the higher levels of Remark \ref{rem:chain}, the constants being again the successive gaps.

\section{The second order CKN inequality: proofs of Theorems \ref{thm:prod2nd} and \ref{thm:2nd}}\label{sec:2nd}

Throughout this section $a>-1$ and $b=-1-a$, so that $1+a-b=2+2a>0$. We use the spherical harmonics decomposition
\[
f(x)=\sum_{k=0}^\infty\sum_{l=1}^{\dim\cH_k}f_{kl}(r)\phi_{kl}(\sigma),\qquad r=|x|,\ \sigma=\frac{x}{|x|},
\]
where $\{\phi_{kl}\}_l$ is an $L^2(\bS^{N-1})$-orthonormal basis of $\cH_k$ and $c_k:=k(k+N-2)$ is the eigenvalue of $-\Delta_{\bS^{N-1}}$ on $\cH_k$. By orthogonality, the deficit $\delta_{2,a}(f)$ equals
\begin{align}\label{decomp}
\sum_{k,l}\Bigg[&\int_0^\infty\Big|f_{kl}''+\frac{N-1}rf_{kl}'-c_k\frac{f_{kl}}{r^2}\Big|^2r^{N-2a-1}dr\notag\\
&+\int_0^\infty\Big(|f_{kl}'|^2+c_k\frac{|f_{kl}|^2}{r^2}\Big)\big(r^{N-2b-1}-(N+4a+2)r^{N-1}\big)dr\Bigg],
\end{align}
and, since the members of $\mathcal G_1$ and $\mathcal G_2$ are gradients of functions whose spherical modes are of degree $0$ and $1$ only, both distances on the right-hand side of \eqref{stab2nd} split over the modes as well (the gradients of functions in different spherical modes are $L^2$-orthogonal, as $\int\nabla(f_{kl}\phi_{kl})\cdot\nabla(f_{k'l'}\phi_{k'l'})=-\int f_{kl}\phi_{kl}\Delta(f_{k'l'}\phi_{k'l'})=0$ for $(k,l)\ne(k',l')$). We treat the radial mode, the mode $k=1$ and the modes $k\ge2$ separately.

\subsection{The radial mode}
Let $f_{01}$ be the radial mode of $f$. We complete the square. Expanding the right-hand side of \eqref{radsq} below, the squared terms are $\int|f_{01}''|^2r^{N-2a-1}$, $(1+2a)^2\int|f_{01}'|^2r^{N-2a-3}$ and $\int|f_{01}'|^2r^{N+2a+1}=\int|f_{01}'|^2r^{N-2b-1}$ (since $-2b=2+2a$), and the cross terms are
\begin{align*}
-(1+2a)\int_0^\infty(|f_{01}'|^2)'r^{N-2a-2}dr&=(1+2a)(N-2a-2)\int_0^\infty|f_{01}'|^2r^{N-2a-3}dr,\\
\int_0^\infty(|f_{01}'|^2)'r^{N}dr&=-N\int_0^\infty|f_{01}'|^2r^{N-1}dr
\end{align*}
and $-2(1+2a)\int|f_{01}'|^2r^{N-1}$. On the other hand $\int|f_{01}''+\frac{N-1}rf_{01}'|^2r^{N-2a-1}=\int|f_{01}''|^2r^{N-2a-1}+(N-1)(1+2a)\int|f_{01}'|^2r^{N-2a-3}$ by \eqref{iden1} with $\gamma=0$. Comparing the coefficients of $\int|f_{01}'|^2r^{N-2a-3}$, namely $(1+2a)^2+(1+2a)(N-2a-2)=(1+2a)(N-1)$, and of $\int|f_{01}'|^2r^{N-1}$, namely $-N-2(1+2a)=-(N+4a+2)$, we obtain
\begin{align}\label{radsq}
\int_0^\infty\Big|f_{01}''+\frac{N-1}rf_{01}'\Big|^2r^{N-2a-1}dr&+\int_0^\infty|f_{01}'|^2r^{N-2b-1}dr-(N+4a+2)\int_0^\infty|f_{01}'|^2r^{N-1}dr\notag\\
&=\int_0^\infty\Big|f_{01}''-\frac{1+2a}{r}f_{01}'+r^{1+2a}f_{01}'\Big|^2r^{N-2a-1}dr ,
\end{align}
Now write $f_{01}'(r)=r^{1+2a}e^{-r^{2+2a}/(2+2a)}v(r)$. Then
\[
f_{01}''=\frac{1+2a}{r}f_{01}'-r^{1+2a}f_{01}'+r^{1+2a}e^{-\frac{r^{2+2a}}{2+2a}}v'
\]
(the derivative of $r^{1+2a}$ produces the first term, that of the exponential the second), so that the square in \eqref{radsq} is exactly $|r^{1+2a}e^{-r^{2+2a}/(2+2a)}v'|^2$, and \eqref{radsq} becomes
\[
\int_0^\infty|v'|^2r^{N+2a+1}e^{-\frac{r^{2+2a}}{1+a}}dr .
\]
This is the left-hand side of \eqref{GP2a} with $\beta=N$ (note $N+4a+2>0$ as $N\ge2$, $a>-1$). Hence, by Corollary \ref{cor:GP}(ii),
\begin{align*}
\int_0^\infty|v'|^2r^{N+2a+1}e^{-\frac{r^{2+2a}}{1+a}}dr&\ge4(1+a)\inf_{c\in\R}\int_0^\infty|v-c|^2r^{N+4a+1}e^{-\frac{r^{2+2a}}{1+a}}dr\\
&=4(1+a)\inf_{c\in\R}\int_0^\infty\Big|f_{01}'-c\,r^{1+2a}e^{-\frac{r^{2+2a}}{2+2a}}\Big|^2r^{N-1}dr\\
&=4(1+a)\inf_{c\in\R}\int_0^\infty\big|f_{01}'-c\,\varphi_0'\big|^2r^{N-1}dr ,
\end{align*}
since $\varphi_0'(r)=-r^{1+2a}e^{-r^{2+2a}/(2+2a)}$. Using also the second-level term of Corollary \ref{cor:GP}, and observing that $r^{1+2a}e^{-\frac{r^{2+2a}}{2+2a}}(c_0+c_1r^{2+2a})$ is exactly the derivative of $(c_0'+c_1'r^{2+2a})\varphi_0$ for suitable $(c_0',c_1')$ (a two-dimensional space in both cases), we have proved
\begin{equation}\label{radial}
\text{radial part of \eqref{decomp}}\ \ge\ 4(1+a)\,d_0^2+4(1+a)\,d_0'^2,
\end{equation}
where
\[
d_0^2:=\inf_{c}\int_0^\infty|f_{01}'-c\varphi_0'|^2r^{N-1}dr,\qquad d_0'^2:=\inf_{c_0,c_1}\int_0^\infty\big|f_{01}'-\big((c_0+c_1r^{2+2a})\varphi_0\big)'\big|^2r^{N-1}dr\le d_0^2 .
\] Both constants are sharp, being the first two levels of Theorem \ref{TA}: equality holds in \eqref{radial} for $v=a_0+a_1r^{2+2a}$ (first term only) and for $v=a_0+a_1r^{2+2a}+a_2r^{4+4a}$ (both terms).

\subsection{The modes $k\ge1$}
Fix $k\ge1$ and write $f_{kl}(r)=r^kg(r)$, $g(r)=h(s)$ with $s=\frac{r^{2+2a}}{2+2a}$; note that $\gamma(m+\gamma-2)=k(k+N-2)=c_k$ for $\gamma=k$, $m=N$, which is why the power $r^k$ is the right one for the mode $k$. Applying \eqref{keyiden} with $m=N$, $\gamma=k$ and $1+a+b=0$, so that $\kappa=2+2a$ and the third integral on its right-hand side disappears, we get, with $\theta=\frac{N+2k+a-3b-1}{2+2a}=\frac{N+2k+4a+2}{2+2a}$ and
\begin{equation}\label{mueps}
\mu:=\theta-1=\frac{N+2k+2a}{2+2a},\qquad \epsilon:=\frac{k}{2+2a},
\end{equation}
the identity
\begin{align*}
\int_0^\infty\Big|f_{kl}''+\frac{N-1}rf_{kl}'-c_k\frac{f_{kl}}{r^2}\Big|^2r^{N-2a-1}dr&+\int_0^\infty\Big(|f_{kl}'|^2+c_k\frac{|f_{kl}|^2}{r^2}\Big)r^{N-2b-1}dr\\
&=(2+2a)^{\theta}\big(Q[h]+P_1[h]\big),
\end{align*}
where in the $|h|^2$-term we used $\frac{2kb}{(2+2a)^2}=-\frac{k}{2+2a}=-\epsilon$. Moreover, by \eqref{iden2} with $b=0$ and the change of variable,
\[
\int_0^\infty\Big(|f_{kl}'|^2+c_k\frac{|f_{kl}|^2}{r^2}\Big)r^{N-1}dr=\int_0^\infty|g'|^2r^{N+2k-1}dr=(2+2a)^{\theta-1}P_0[h].
\]
Written out with the explicit exponents, the two identities read
\begin{align*}
&\int_0^\infty\Big|f_{kl}''+\frac{N-1}rf_{kl}'-c_k\frac{f_{kl}}{r^2}\Big|^2r^{N-2a-1}dr+\int_0^\infty\Big(|f_{kl}'|^2+c_k\frac{|f_{kl}|^2}{r^2}\Big)r^{N-2b-1}dr\\
&\quad=(2+2a)^{\frac{N+2k+4a+2}{2+2a}}\Big[\int_0^\infty|h''|^2s^{\frac{N+2k+2a}{2+2a}+1}ds+\int_0^\infty|h'|^2s^{\frac{N+2k+2a}{2+2a}+1}ds\\
&\qquad\qquad\qquad\qquad\qquad-\frac{k}{2+2a}\int_0^\infty|h|^2s^{\frac{N+2k+2a}{2+2a}-1}ds\Big],\\
&\int_0^\infty\Big(|f_{kl}'|^2+c_k\frac{|f_{kl}|^2}{r^2}\Big)r^{N-1}dr=\int_0^\infty|g'|^2r^{N+2k-1}dr\\
&\quad=(2+2a)^{\frac{N+2k+2a}{2+2a}}\int_0^\infty|h'|^2s^{\frac{N+2k+2a}{2+2a}}ds .
\end{align*}
Note that $\mu^2-4\epsilon=\frac{(N+2k+2a)^2-8k(1+a)}{(2+2a)^2}>0$. Therefore, by Theorem \ref{thm:1d} with $\mu=\frac{N+2k+2a}{2+2a}$ and $\epsilon=\frac{k}{2+2a}$, i.e.\ by Hamamoto's inequality
\begin{align*}
&\int_0^\infty|h''|^2s^{\mu+1}ds+\int_0^\infty|h'|^2s^{\mu+1}ds-\frac{k}{2+2a}\int_0^\infty|h|^2s^{\mu-1}ds\\
&\qquad\ge\ \Big(\sqrt{\Big(\frac{N+2k+2a}{2+2a}\Big)^2-\frac{2k}{1+a}}+1\Big)\int_0^\infty|h'|^2s^{\mu}ds ,
\end{align*}
and using $(2+2a)\sqrt{\big(\frac{N+2k+2a}{2+2a}\big)^2-\frac{2k}{1+a}}=\sqrt{(N+2k+2a)^2-8k(1+a)}$, we get
\begin{align*}
&\int_0^\infty\Big|f_{kl}''+\frac{N-1}rf_{kl}'-c_k\frac{f_{kl}}{r^2}\Big|^2r^{N-2a-1}dr+\int_0^\infty\Big(|f_{kl}'|^2+c_k\frac{|f_{kl}|^2}{r^2}\Big)r^{N-2b-1}dr\\
&\qquad\ge(2+2a)\Big(1+\sqrt{\mu^2-4\epsilon}\Big)\int_0^\infty\Big(|f_{kl}'|^2+c_k\frac{|f_{kl}|^2}{r^2}\Big)r^{N-1}dr\\
&\qquad=\Big(\sqrt{(N+2k+2a)^2-8k(1+a)}+2+2a\Big)\int_0^\infty\Big(|f_{kl}'|^2+c_k\frac{|f_{kl}|^2}{r^2}\Big)r^{N-1}dr ,
\end{align*}
and consequently, for the $(k,l)$-term of \eqref{decomp},
\begin{align}\label{modek}
\text{$(k,l)$-term of \eqref{decomp}}\ &\ge\ C(N,a,k)\int_0^\infty\Big(|f_{kl}'|^2+c_k\frac{|f_{kl}|^2}{r^2}\Big)r^{N-1}dr\notag\\
&=C(N,a,k)\int_{\R^N}|\nabla(f_{kl}\phi_{kl})|^2dx ,
\end{align}
where
\begin{align*}
C(N,a,k)&:=\sqrt{(N+2k+2a)^2-8k(1+a)}-(N+2a)\\
&=\sqrt{(N+2a)^2+4k^2+4k(N-2)}-(N+2a) ,
\end{align*}
i.e.\ $C(N,a,k)=C_k(N,a)$ as defined in \eqref{C1}, which is increasing in $k$.

For $k=1$ we can do better. By the remainder term in Theorem \ref{thm:1d}, applied with the parameters \eqref{mueps} for $k=1$ (so that $\mu=\mu_1$, $\epsilon=\epsilon_1$ and $b=b_1$ are the quantities in \eqref{h1}, and $h_1=f_0$ in the notation \eqref{bf0}),
\[
Q[h]+P_1[h]-\Big(1+\sqrt{\mu_1^2-4\epsilon_1}\Big)P_0[h]\ \ge\ 2\inf_{c\in\R}P_0[h-c\,h_1] ,
\]
and since the identity $\int_0^\infty(|f_{1l}'|^2+c_1\frac{|f_{1l}|^2}{r^2})r^{N-1}dr=(2+2a)^{\theta-1}P_0[h]$ is quadratic in the profile, $(2+2a)^{\theta-1}P_0[h-ch_1]=\int_{\R^N}|\nabla(f_{1l}\phi_{1l})-c\,\nabla(r\,h_1(s)\phi_{1l})|^2dx$. Therefore, in addition to \eqref{modek},
\begin{align}\label{mode1}
&\text{$(1,l)$-term of \eqref{decomp}}\notag\\ \ge\ &C_1(N,a)\int_{\R^N}|\nabla(f_{1l}\phi_{1l})|^2dx\notag\\ &+4(1+a)\inf_{c\in\R}\int_{\R^N}\big|\nabla(f_{1l}\phi_{1l})-c\,\nabla(r\,h_1(s)\phi_{1l})\big|^2dx ,
\end{align}
where $4(1+a)=2(2+2a)$ comes from $(2+2a)^{\theta}\cdot2=2(2+2a)\cdot(2+2a)^{\theta-1}$.

\subsection{Conclusion}
Write $d_1^2:=\inf_{\phi\in\cH_1}\int|\nabla f_{(1)}-\nabla(r h_1(s)\phi)|^2$ for the degree-one part $f_{(1)}=\sum_lf_{1l}\phi_{1l}$ of $f$, and $\|\nabla f_{(k)}\|^2:=\int|\nabla f_{(k)}|^2$ for the degree-$k$ parts. By the orthogonality described after \eqref{decomp},
\[
\dist(\nabla f,\mathcal G_1)^2=d_0^2+\sum_{k\ge1}\|\nabla f_{(k)}\|^2,\qquad
\dist(\nabla f,\mathcal G_2)^2=d_0'^2+d_1^2+\sum_{k\ge2}\|\nabla f_{(k)}\|^2 ,
\]
and $\delta_{2,a}(f)$ is bounded from below, by \eqref{radial}, \eqref{mode1} and \eqref{modek}, by
\[
4(1+a)d_0^2+4(1+a)d_0'^2+C_1(N,a)\|\nabla f_{(1)}\|^2+4(1+a)d_1^2+\sum_{k\ge2}C_k(N,a)\|\nabla f_{(k)}\|^2 .
\]
We compare block by block with $\Theta_1\dist(\nabla f,\mathcal G_1)^2+\Theta_2\dist(\nabla f,\mathcal G_2)^2$. Radial block: $\Theta_1d_0^2+\Theta_2d_0'^2\le4(1+a)d_0^2+4(1+a)d_0'^2$ because $\Theta_1\le4(1+a)$, $\Theta_1+\Theta_2\le8(1+a)$ and $d_0'\le d_0$. Block $k=1$: $\Theta_1\|\nabla f_{(1)}\|^2+\Theta_2d_1^2\le C_1\|\nabla f_{(1)}\|^2+4(1+a)d_1^2$ because $\Theta_1\le C_1$, $\Theta_1+\Theta_2\le C_1+4(1+a)$ and $d_1\le\|\nabla f_{(1)}\|$. Blocks $k\ge2$: $(\Theta_1+\Theta_2)\|\nabla f_{(k)}\|^2\le C_2\|\nabla f_{(k)}\|^2\le C_k\|\nabla f_{(k)}\|^2$. This proves \eqref{stab2nd}, and \eqref{stab2nd1} follows by dropping the second term.

For the sharpness of $\Theta_1$: if $\Theta_1=4(1+a)$ take $f$ radial with $v=r^{2+2a}$, for which \eqref{radial} is an equality with $d_0'=0$; if $\Theta_1=C_1$ take $f=f_{11}\phi_{11}$ with $h=h_1$, for which \eqref{mode1} is an equality with $d_1=0$. In both cases only the first term of \eqref{stab2nd} is nonzero and equality holds. For the sharpness of $\Theta_2$ we use the test function realizing the minimum in the definition of $\Theta_1+\Theta_2$: a radial $f$ whose $v$ is the Laguerre polynomial of degree two in $r^{2+2a}$, i.e.\ $v=a_0+a_1r^{2+2a}+a_2r^{4+4a}$ orthogonal to the two lower levels (deficit $8(1+a)$ times $\int|\nabla f|^2$, both distances equal to $\int|\nabla f|^2$); or $f=f_{11}\phi_{11}$ with $h=f_1$ the second-level function \eqref{f1} of Theorem \ref{thm:1d} (deficit $(C_1+4(1+a))\int|\nabla f|^2$, both distances equal to $\int|\nabla f|^2$); or $f=f_{21}\phi_{21}$ with $h=f_0$ the extremal of Theorem \ref{thm:1d} for the parameters of $k=2$ (deficit $C_2\int|\nabla f|^2$, both distances equal to $\int|\nabla f|^2$). In each case equality holds in \eqref{stab2nd} with the second term nonzero. This completes the proof of Theorem \ref{thm:2nd}. \qed

\subsection{Proof of Theorem \ref{thm:prod2nd}}
We apply Lemma \ref{lem:transfer} with $p=2+2a$. For \eqref{prod2nd} we use the dilation $f_\lambda(x):=\lambda^{N/2-1}f(\lambda x)$, for which $\nabla f_\lambda=(\nabla f)_\lambda$ in the notation $\bU_\lambda=\lambda^{N/2}\bU(\lambda\,\cdot)$ of Section \ref{sec:prod}, so that $\int|\nabla f_\lambda|^2=\int|\nabla f|^2$, while
\[
\int|\Delta f_\lambda|^2|x|^{-2a}=\lambda^{2+2a}\int|\Delta f|^2|x|^{-2a},\qquad\int|\nabla f_\lambda|^2|x|^{-2b}=\lambda^{-2-2a}\int|\nabla f|^2|x|^{-2b}
\]
(recall $b=-1-a$). The choice $\lambda_*=(B/A)^{1/(4+4a)}$ balances the two integrals, so that $\delta_{2,a}(f_{\lambda_*})=2\delta_{1,a}(f)$, and Theorem \ref{thm:2nd} applied to $f_{\lambda_*}$ gives \eqref{prod2nd}, with the distances to $(\cG_j)_{1/\lambda_*}$ in the stronger form. Note that the dilates of $\nabla\varphi_0$ are the gradients $\nabla\big(\lambda^{N/2-1}\varphi_0(\lambda\,\cdot)\big)$, so $\cG_1^{\rm dil}$ is exactly the set of gradients of the extremals $c\,\varphi_0(\lambda\,\cdot)$ of \eqref{2ndCKN}.

For the sharpness in Theorem \ref{thm:prod2nd}, let $a\ge a_*(N)$, so that $\Theta_1=C_1(N,a)$, and take $f=r\,h_1(s)\phi_1$. This function minimizes the ratio of $\delta_{2,a}$ to the dilation invariant quantity $\int|\nabla f|^2$ over the first spherical mode, which is invariant under dilations; hence it is balanced, i.e.\ $\frac{d}{d\lambda}\delta_{2,a}(f_\lambda)|_{\lambda=1}=0$, and $\delta_{1,a}(f)=\frac12\delta_{2,a}(f)=\frac12C_1(N,a)\int|\nabla f|^2$. Moreover $\nabla f$ is orthogonal to all the radial fields, hence $\dist(\nabla f,\cG_1^{\rm dil})^2=\int|\nabla f|^2$, and equality holds in the first term of \eqref{prod2nd}. \qed

The restriction $a\ge a_*(N)$ in the sharpness statement is not a technical one. For $a<a_*(N)$ the constant $\Theta_1=4(1+a)$ of Theorem \ref{thm:2nd} is attained in the radial mode by the first Laguerre level, and this level is, up to a multiple of $\nabla\varphi_0$, the infinitesimal generator $\frac{d}{d\lambda}\nabla(\lambda^{N/2-1}\varphi_0(\lambda\,\cdot))|_{\lambda=1}$ of the dilations, along which the product deficit $\delta_{1,a}$ vanishes. A second order expansion of $\delta_{1,a}$ at $\varphi_0$ shows that the first Laguerre level does not contribute to the product deficit near the set of extremals, so that the local stability constant of $\delta_{1,a}$ near $\cG_1^{\rm dil}$ is $\frac12\min\{8(1+a),C_1(N,a)\}$. We do not know the sharp global constant of \eqref{prod2nd} for $a<a_*(N)$; the estimate \eqref{prod2nd} holds with the constant $\frac12\Theta_1=2(1+a)$.

At $a=0$, \eqref{2ndCKN} is the second order HUP, $\varphi_0(x)=e^{-|x|^2/2}$, $\mathcal G_1$ is spanned by $\nabla e^{-|x|^2/2}=-xe^{-|x|^2/2}$, and Theorems \ref{thm:prod2nd} and \ref{thm:2nd} give the sharp stability of the second order HUP with the constants $\Theta_1=\sqrt{N^2+4N-4}-N$ (at a fixed scale) and $\frac12\Theta_1$ (in product form), attained in the first spherical mode: the extremizing test functions are $f=r\,h_1(r^2/2)\phi_1(\sigma)$ with $h_1=e^{-s}{}_1F_1(b_1,\frac{N+2}2,s)$, $b_1=\frac{N+2-\sqrt{N^2+4N-4}}{4}$. In the language of curl-free fields, this is the sharp stability of the curl-free HUP of \cite{CFL22}, for $\bU=\nabla f$; the reader may compare with the results of \cite{DoLL26, DN25, HY25}, obtained by different methods. The second term of the chain is governed by $\Theta_2=\min\{8,\sqrt{N^2+4N-4}-N+4,\sqrt{N^2+8N}-N\}-\Theta_1$, and one checks that the minimum is always the third competitor, i.e.\ $\Theta_1+\Theta_2=C_2(N,0)=\sqrt{N^2+8N}-N$, the constant of the second spherical mode.

\medskip
\noindent\textbf{Acknowledgements.} A. X. Do and G. Lu were partially supported by grants from the Simons Foundation. A. T. Duong and V. H. Nguyen were supported by the Vietnam National Foundation for Science and Technology Development (NAFOSTED) under grant number 101.02-2025.33. N. Lam was partially supported by an NSERC Discovery Grant.

\end{document}